\documentclass[10pt, reqno]{amsart}
\usepackage{graphicx, amssymb, amsmath, amsthm, color, slashed, cite}
\numberwithin{equation}{section}

\usepackage{hyperref}

\let\Re=\undefined\DeclareMathOperator*{\Re}{Re}
\let\Im=\undefined\DeclareMathOperator*{\Im}{Im}

\newcommand{\R}{\mathbb{R}}
\newcommand{\C}{\mathbb{C}}

\newcommand{\eps}{\varepsilon}

\newtheorem{theorem}{Theorem}[section]

\newtheorem{lemma}[theorem]{Lemma}

\newtheorem{corollary}[theorem]{Corollary}

\newtheorem{proposition}[theorem]{Proposition}

\theoremstyle{definition}

\theoremstyle{remark}

\newcommand{\qtq}[1]{\quad\text{#1}\quad}

\begin{document}

\title[NLS with repulsive potentials]{Threshold scattering for the focusing NLS with a repulsive potential}

\author[Miao]{Changxing Miao}
\address{Institute for Applied Physics and Computational Mathematics, Beijing, China}
\email{miao\_changxing@iapcm.ac.cn}

\author[Murphy]{Jason Murphy} 
\address{Department of Mathematics \& Statistics, Missouri S\&T, USA}
\email{jason.murphy@mst.edu}

\author[Zheng]{Jiqiang Zheng}
\address{Institute for Applied Physics and Computational Mathematics, Beijing, China}
\email{zhengjiqiang@gmail.com}

\begin{abstract} We adapt the arguments in the recent work of Duyckaerts, Landoulsi, and Roudenko to establish a scattering result at the sharp threshold for the $3d$ focusing cubic NLS with a repulsive potential.  We treat both the case of short-range potentials as previously considered in the work of Hong, as well as the inverse-square potential, previously considered in the work of the authors. 
\end{abstract}

\maketitle


\section{Introduction}

We consider the focusing cubic nonlinear Schr\"odinger equation with an external potential in three space dimensions.  This equation takes the form
\begin{equation}\label{nls}\tag{$\text{NLS}_V$}
i\partial_t u = (-\Delta + V)u - |u|^2 u,
\end{equation}
where $u:\R\times\R^3\to\C$ is a complex-valued function and $V:\R^3\to\R$.  We focus our attention on nonnegative, repulsive potentials.  In particular, our results will address the following:
\begin{itemize}
\item[(i)] We first consider potentials $V:\R^3\to\R$ as in the work of Hong \cite{Hong}.  This firstly requires the specific decay/regularity conditions 
\begin{equation}\label{V1}
V\in L^{\frac32},\quad x\cdot\nabla V\in L^{\frac32},\qtq{and} \sup_{x\in\R^3}\int \frac{|V(y)|}{|x-y|}\,dy<\infty.
\end{equation}
In addition, we require that $V$ be nonnegative and repulsive, that is:
\begin{equation}\label{V2}
V\geq 0 \qtq{and}x\cdot \nabla V\leq 0.
\end{equation}
\item[(ii)] We next consider the case of a repulsive inverse-square potential, i.e.
\begin{equation}\label{V3}
V(x) = a|x|^{-2}\qtq{with}a>0,
\end{equation}
which we have studied in previous works \cite{KMVZ, KMVZZ, LMM}.  This represents a limiting case of the potentials considered in (i), in the sense that \eqref{V2} is satisfied but the decay requirements in \eqref{V1} barely fail.
\end{itemize}

The works \cite{Hong, KMVZ} considered the problem of finding the sharp scattering threshold for \eqref{nls} with external potentials as in (i) and (ii) above.  In fact, these works show that for nonnegative, repulsive potentials, one obtains scattering below the same threshold arising in the setting of the standard cubic NLS
\begin{equation}\label{nls0}\tag{$\text{NLS}_0$}
i\partial_t u = -\Delta u -|u|^2 u.
\end{equation}
To state the results precisely, we first introduce the conserved \emph{mass} and \emph{energy} of solutions to \eqref{nls}, given by
\[
M(u) = \tfrac12\int |u|^2\,dx \qtq{and} E_V(u) = \int \tfrac12|\nabla u|^2 + \tfrac12 V(x)|u|^2 - \tfrac14|u|^4\,dx,
\]
respectively.  Next, we denote by $Q$ the ground state soliton for ($\text{NLS}_0$), that is, the unique nonnegative, decaying, radial solution to
\begin{equation}\label{elliptic}
-Q+\Delta Q + Q^3 = 0.
\end{equation}
Finally, we define the Sobolev space adapted to 
\[
H:=-\Delta + V
\]
by
\[
\|u\|_{\dot H_V^1}^2 = \langle Hu,u\rangle = \int |\nabla u|^2 + V|u|^2\,dx.  
\]
The scattering results of \cite{Hong, KMVZ} may then be stated as follows:
\begin{theorem}[Sub-threshold scattering, \cite{Hong, KMVZ}]\label{T:sub} Let $V$ satisfy \eqref{V1}--\eqref{V2} or \eqref{V3}.  Suppose $u_0 \in H^1(\R^3)$ obeys
\begin{equation}\label{old-sub-threshold}
M(u_0)E_V(u_0)<M(Q)E_0(Q) \qtq{and} \|u_0\|_{L^2}\| u_0\|_{\dot H_V^1} < \|Q\|_{L^2} \|Q\|_{\dot H^1}. 
\end{equation}
Then the solution to \eqref{nls} with initial data $u_0$ is global in time and obeys 
\[
\|u\|_{L_{t,x}^5(\R\times\R^3)}<C\bigl(M(Q)E_0(Q)-M(u_0)E_V(u_0)\bigr)
\]
for some function $C:(0,M(Q)E_0(Q))\to(0,\infty)$.  Consequently, $u$ scatters in both time directions; that is, there exist $u_\pm\in H^1$ so that
\[
\lim_{t\to\pm\infty} \|u(t)-e^{-itH}u_\pm\|_{H^1} = 0. 
\] 
\end{theorem}

In the result above, we may also include the case $V\equiv 0$, in which case we recover the results of \cite{HR, DHR}.

Theorem~\ref{T:sub} is sharp in the sense that the constant $C(\cdot)$ necessarily diverges as one approaches the threshold.  For \eqref{nls0}, this follows immediately from the existence of the solitary wave solution $u(t,x)=e^{it}Q(x)$.  For \eqref{nls}, one can derive this fact by considering a sequence of solutions constructed by approximations built from translations of the solitary wave solution. In particular, we have the following (see \cite[Theorem~1.5]{KMVZ}).

\begin{theorem}[Failure of uniform bounds at the threshold] Let $V$ satisfy \eqref{V1}--\eqref{V2} or \eqref{V3}. Then there exist global solutions $u_n$ to \eqref{nls} so that
\[
M(u_n)E_V(u_n)\nearrow M(Q)E_0(Q),\quad \|u_n(0)\|_{L^2}\|u_n(0)\|_{\dot H^1_V}\nearrow \|Q\|_{L^2}\|Q\|_{\dot H^1}, 
\] 
and
\[
\lim_{n\to\infty}\|u\|_{L_{t,x}^5(\R\times\R^3)}=\infty. 
\]
\end{theorem}

In this work, we consider solutions at the sharp mass-energy threshold and establish the following scattering result. 
\begin{theorem}[Threshold scattering]\label{T} Let $V$ satisfy \eqref{V1}--\eqref{V2} or \eqref{V3}.  Suppose $u_0\in H^1(\R^3)$ satisfies
\begin{equation}\label{SUBTH}
M(u_0)E_V(u_0)=M(Q)E_0(Q)\qtq{and} \|u_0\|_{L^2}\|u_0\|_{\dot H_V^1}<\|Q\|_{L^2}\|Q\|_{\dot H^1}.
\end{equation}
Then the corresponding solution $u$ to \eqref{nls} is global, with $u\in L_{t,x}^5(\R\times\R^3)$.  In particular, $u$ scatters in $H^1$.
\end{theorem}

To put this result in context, we first recall the work of \cite{DR} on threshold solutions for the standard cubic equation \eqref{nls0}.  In this work, all possible behaviors of solutions with initial data obeying \eqref{SUBTH} (with $V\equiv0$) are classified.  In contrast to Theorem~\ref{T}, solutions do not necessarily scatter in both time directions.  Instead, solutions may converge to another special solution $Q^{-}$ to $(\text{NLS}_0)$ as $t\to\infty$ or $t\to-\infty$.  (In fact, this is only one part of the main result in \cite{DR}, but we focus here on the aspects of \cite{DR} most closely related to the present work.) 

We next recall the works of \cite{KVZ, DLR}, which considered the cubic NLS in the exterior of a convex obstacle.  In \cite{KVZ}, the authors established sub-threshold scattering (the analogue of Theorem~\ref{T:sub} above), with the sharp threshold again given by \eqref{SUBTH}.  The recent work \cite{DLR} subsequently established a scattering result at the sharp mass-energy threshold in this setting (as in Theorem~\ref{T}).  In particular, our result is an analogue of the result of \cite{DLR} in the setting of NLS with non-negative, repulsive potentials.  Accordingly, our methods are inspired by the arguments given in that work, which are in turn related to works such as \cite{DR, DM}. 

Finally, let us mention the related work \cite{YZZ}, which considered the classification of dynamics of threshold solutions for the energy-critical NLS with an inverse-square potential.  In this case, the authors restricted attention to the case of an \emph{attractive} potential and established results in the spirit of \cite{DM}.

In the rest of the introduction, we will describe the main ideas of the proof of Theorem~\ref{T} and provide an outline of the structure of the paper.  As mentioned above, the proof follows a similar strategy to the ones appearing in \cite{DLR, DR, DM}

The proof of Theorem~\ref{T} is by contradiction.  The first step is Proposition~\ref{P:compact}, in which we show that if the theorem fails, then we may find a forward-global, $H^1$-bounded, `compact' solution at the threshold with infinite $L_{t,x}^5$ norm on $[0,\infty)$.  More precisely, the orbit of this solution is pre-compact in $H^1$ modulo some time-dependent spatial center $x(t)$.  The proof of Proposition~\ref{P:compact} essentially follows the standard path of concentration-compactness.  The key to obtaining compactness is to prevent `dichotomy', which we do primarily by appealing to the sub-threshold scattering theory, i.e. Theorem~\ref{T:sub}. 

The key to obtaining a contradiction is now to analyze the possible behavior of $x(t)$, or, more precisely, to show that no behavior for $x(t)$ is possible!  Indeed, we show that if $x(t)$ is bounded, then it must be unbounded (Section~\ref{BtoUB}); and if $x(t)$ is unbounded, then it must be bounded (Section~\ref{UBtoB}).  We thus obtain a contradiction and complete the proof of Theorem~\ref{T}.

The analysis of $x(t)$ relies primarily on suitable virial arguments and involves the introduction of the related quantity 
\[
\delta(u(t)) := \int |\nabla Q|^2\,dx - \int |\nabla u|^2 + V|u|^2\,dx,
\]
which, in particular, is dominated by the quantity that arises in the virial identity.  Consequently, if we assume $x(t)$ remains bounded, then the standard localized virial argument implies that $\delta(u(t_n))\to 0$ along some sequence $t_n\to\infty$.  By the sharp Gagliardo--Nirenberg inequality, this forces $u$ to approach the orbit of $Q$ and the potential part of the energy to vanish, which in turn forces $|x(t_n)|\to\infty$. That is, if $x(t)$ is bounded, then $x(t)$ is unbounded.

To prove the reverse implication is more involved.  To begin, we need the converse statement that $|x(t_n)|\to\infty$ implies $\delta(u(t_n))\to0$.  This relies on an approximation argument akin to \cite[Theorem~6.1]{KMVZ}, which shows that \eqref{nls} is well-approximated by the standard cubic equation \eqref{nls0} in the regime $|x|\to\infty$, together with the threshold classification result of \cite{DR}.  Basically, these two results together show that if $|x(t_n)|\to\infty$ but $\delta(u(t_n))\geq c>0$, then the solution cannot blow up its $L_{t,x}^5$-norm. 

Next, we analyze the behavior of the solution in the regime $\delta(u(t))\ll 1$.  As above, we find that the solution approaches the orbit of $Q$.  In particular, we can carry out a `modulation analysis' and obtain a description of the solution of the form
\begin{equation}\label{intro-modulation}
u(t,x) = e^{i\theta(t)}[Q(x-y(t))+g(t,x)],
\end{equation}
where (by choosing the modulation parameters correctly) we can obtain bounds over quantities like $\|g\|_{H^1}$, $|\dot y|$, and the potential part of the energy all in terms of $\delta(u(t))$.  This part of the analysis is essentially independent of the rest of the proof and is relegated to the final section of the paper, Section~\ref{S:modulation}.  While it seems at this moment that one solution is being parametrized by two spatial centers (i.e. $x(t)$ and $y(t)$), we can show that $|x(t)-y(t)|=\mathcal{O}(1)$ and hence simply re-define $x(t)=y(t)$ wherever $\delta(u(t))$ is small. 

The final step is now to run another virial argument in order to gain control over integrals over the form $\int_I \delta(u(t))\,dt$.  We can then use this quantity to control the variation of $x(t)$ on $I$.  Thus, if we can show that this quantity is controlled by a small multiple of $\sup_I |x(t)|$, then we can prove that `$x(t)$ controls itself' and hence remains bounded.   (Altogether, we obtain that if $x(t)$ is unbounded, then $x(t)$ is bounded.)

This time the virial argument is more subtle, due to the fact that we must control error terms in terms of $\delta(u(t))$ itself.   To achieve this, we split the interval into times where $\delta(u(t))$ is small versus where it admits some lower bound.  For times where $\delta(u(t))$ is small, we must `add zero' in a very specific way that lets us exploit the modulation analysis and exhibit terms containing the error term `$g(t)$' in \eqref{intro-modulation} above.  In particular, because $e^{it}Q$ is a solution to \eqref{nls0}, quantities appearing within the virial identities can be seen to vanish when evaluated at $e^{i\theta}[Q(\cdot-y)]$ (see Section~\ref{S:virial}). 

Having sketched the technical aspects of the arguments above, we can also offer the following rough `dynamical' description:  As long as $x(t)$ remains bounded, the virial identity implies that $u(t)$ must move away from the origin (the weighted momentum is increasing).  However, once $u(t)$ moves far enough from the origin, it is drawn towards the orbit of $Q$, at which point it loses momentum and slows down, ultimately limiting the motion of $x(t)$.  The only conclusion is that this type of `compact' solution cannot exist.  

At various points in the paper, we need slightly different arguments to deal with potentials obeying \eqref{V1}--\eqref{V2} versus \eqref{V3}.  As we will see, some approximation arguments (in the regime $|x|\to\infty$) are greatly simplified for potentials obeying \eqref{V1}--\eqref{V2}.  For the virial arguments, the essential property needed is the repulsive assumption, which leads to a term with a good sign in the virial inequality.  However, it turns out that to run the `modulated' virial argument, the general case is much more of a headache than the inverse-square potential.  This is due to the fact that the inverse-square potential preserves the scaling symmetry, which is intimately linked to the virial identity.  In particular, for the inverse-square potential, the quantity arising in the virial identity coincides with the potential part of the energy, whereas in the general case it leads to a term that necessitates some additional estimates on the modulation parameter $|y(t)|$ and the use of explicit exponential estimates for the ground state $Q$.   The interested reader can find a bit more detail on this point in Section~\ref{S:conclusion}.

Finally, let us briefly mention that if one considers Theorem~\ref{T} with a radial assumption, then one obtains the condition $x(t)\equiv 0$ immediately and hence reaches a contradiction by the standard virial argument.  However, this tells us nothing new!  Indeed, it was already shown in \cite[Theorem~1.6]{KMVZ} that radial solutions to \eqref{nls} enjoy a strictly larger scattering threshold than the one appearing in \eqref{SUBTH}. 

The rest of this paper is organized as follows:
\begin{itemize}
\item Section~\ref{S:preliminaries} contains preliminary material.  This includes a discussion of the local theory for \eqref{nls}, stability theory, concentration-compactness results, a bit of variational analysis, and a section on virial identities. 
\item In Section~\ref{S:compact}, we prove Proposition~\ref{P:compact}, which shows that if Theorem~\ref{T} fails, then we may find a compact, nonscattering, threshold solution.  The rest of the paper is devoted to precluding the possibility of such a solution.
\item In Section~\ref{S:impossible}, we rule out the possibility of solutions as in Proposition~\ref{P:compact}, taking for granted the modulation analysis to be carried out in Section~\ref{S:modulation}. We begin in Section~\ref{S:DP} by introducing several `main characters' in the argument and spelling out the relationships between them.  In Section~\ref{BtoUB}, we show that if $x(t)$ is bounded, then it is unbounded; in Section~\ref{UBtoB}, we prove the converse.  In Section~\ref{S:conclusion}, we conclude the proof and make a few technical remarks.
\item Finally, in Section~\ref{S:modulation}, we carry out the modulation analysis, providing a description of the solution at the times when it is close the orbit of $Q$. 
\end{itemize}

\subsection*{Acknowledgments} C. Miao was supported by the National Key Research and Development Program  of China (No. 2020YFA0712900) and NSFC Grant 11831004.  J. Zheng was supported by NSFC Grant 11901041.

\section{Preliminaries}\label{S:preliminaries}
For functions depending only on $t$, we may use $\dot{}$ to denote $\tfrac{d}{dt}$.  We write $B_r(x_0)$ for the ball of radius $r$ centered at $x_0$.  We denote the standard $L^2$ inner product by
\[
\langle f,g\rangle = \int \bar f g\,dx.
\]
We employ the usual $L_t^q L_x^r$ notation for mixed Lebesgue space-time norms. 

We define
\[
\|f\|_{\dot H_V^1}^2 = \langle f,H f\rangle = \int_{\R^3} |\nabla f|^2 + V(x)|f|^2\,dx. 
\]

Under the assumptions \eqref{V1}--\eqref{V2} or \eqref{V3}, we have the following result concerning the equivalence of Sobolev spaces defined in terms of $-\Delta$ and those defined in terms of $H=-\Delta+V$ (see \cite[Lemma~2.6]{Hong} and \cite{KMVZZ2}).
\begin{lemma}[Equivalence of Sobolev spaces]\label{L:EOSS} Suppose $V$ satisfies \eqref{V1}--\eqref{V2} or \eqref{V3}.  Then for all $1<r<\tfrac{3}{s}$, we have
\[
\| H^{s/2}f\|_{L^r} \sim \| |\nabla|^s f\|_{L^r}. 
\]
\end{lemma}

\subsection{Local theory} We recall here the local well-posedness theory and stability theory for \eqref{nls}.  For more details, see \cite{Hong, KMVZ}.

First, we have the following well-posedness result (see \cite[Theorem~2.15]{KMVZ} and \cite[Lemma~2.11]{Hong}).

\begin{proposition}[Well-posedness]\label{P:LWP} Suppose $V$ satisfies \eqref{V1}--\eqref{V2} or \eqref{V3}. 
\begin{itemize}
\item For any initial data $u_0\in H^1$, there exists a unique maximal-lifespan solution $u$ to \eqref{nls}.  Any solution that remains uniformly bounded in $H^1$ throughout its lifespan is global in time.
\item Additionally, if
\[
\|e^{-itH}u_0\|_{L_{t,x}^5((0,\infty)\times\R^3)}
\]
is sufficiently small, then the solution with data $u_0$ is forward-global and obeys $L_{t,x}^5$-bounds forward in time.  
\item More generally, any $H^1$ solution that remains uniformly bounded in $L_{t,x}^5$ throughout its lifespan is global, and global $L_{t,x}^5$-bounds imply scattering.
\item Finally, given any $\psi\in H^1$ we may construct a solution to \eqref{nls} on some interval $(T,\infty)$ that scatters to $\psi$ in $H^1$. 
\end{itemize}
Analogous statements hold backward in time, as well.
\end{proposition}

We will also need the the following stability result for \eqref{nls}. The primary role of this result will be to transfer space-time bounds from various approximate solutions to true solutions to \eqref{nls}.  

If $V$ satisfies \eqref{V1}--\eqref{V2}, then a suitable stability result appears in \cite[Lemma~2.14]{Hong}.  As in that work, we will utilize `exotic' Strichartz spaces (see \cite{Foschi} and \cite[Lemma~2.4]{Hong}) for the purpose of estimating errors (see the $L_t^4 L_x^{\frac65}$ component of the norm in \eqref{def:N} below).  If $V$ instead satisfies \eqref{V3}, then a suitable result appears in \cite[Theorem~2.17]{KMVZ}, where errors are estimated in the more `traditional' Strichartz spaces with derivatives (see the first three components in \eqref{def:N} below).  Here the only subtle point is to work in spaces in which one has equivalence of Sobolev spaces, so that one can freely interchange $H^{s/2}$ (which commutes with the equation) with $|\nabla|^s$ (for which one can use fractional calculus estimates).

Introducing the notation
\begin{equation}\label{def:N}
\dot N^{\frac12}(I):=L_t^{\frac{10}{7}}\dot H_x^{\frac12,\frac{10}{7}} + L_t^{\frac{5}{3}}\dot H_x^{\frac12,\frac{30}{23}}+L_t^1 \dot H_x^{\frac12}+ L_t^4 L_x^{\frac65}
\end{equation}
and
\[
\dot S^s(I) = L_t^\infty \dot H_V^s\cap L_t^2 \dot H_V^{s,6},\quad S^s(I)= L_t^\infty  H_V^s\cap L_t^2  H_V^{s,6},
\]
with all space-time norms over $I\times\R^3$, we obtain the following: 

\begin{proposition}[Stability]\label{P:stab}  Let $V$ satisfy \eqref{V1}--\eqref{V2} or \eqref{V3}.  Suppose $\tilde v:I\times\R^3\to\C$ solves
\[
(i\partial_t -H)\tilde v = -|\tilde v|^2 \tilde v + e,\quad v(t_0)=\tilde v_0\in H^1,
\]
where $e:I\times\R^3\to\C$.  Let $v_0\in H^1$, and suppose
\[
\|v_0\|_{H^1} + \|\tilde{v}_0\|_{H^1} \leq E \qtq{and} \| \tilde{v}\|_{L_{t,x}^5(I\times\R^3)} \leq L
\]
for some $E,L>0$.  There exists $\eps_0=\eps_0(E,L)>0$ such that if $0<\eps<\eps_0$ and
\begin{equation}\label{stab:small}
\|\tilde v_0 - v_0 \|_{\dot H^{\frac12}} + \|e\|_{\dot N^{\frac12}(I)} < \eps,
\end{equation}
then there exists a solution $v:I\times\R^3\to\C$ to \eqref{nls} with $v(t_0)=v_0$ satisfying
\begin{align}
\label{E:stab}
&\|v-\tilde v\|_{\dot S^{\frac12}(I)} \lesssim_{E,L} \eps, \\
\label{E:stab-bound}
&\| v\|_{S^1(I)}\lesssim_{E,L} 1.
\end{align}
\end{proposition}

\subsection{Concentration-compactness}

In this section we import a linear profile decomposition adapted to the $H^1\to L_{t,x}^5$ Strichartz estimate for $e^{-itH}$.  This decomposition plays a key role in establishing compactness for nonscattering solutions in Section~\ref{S:compact}. For potentials satisfying \eqref{V1}--\eqref{V2}, the following result appears as Proposition~5.1 in \cite{Hong}; for potentials satisfying \eqref{V3}, the result appears as Proposition~5.1 in \cite{KMVZ}.

\begin{proposition}[Linear profile decomposition]\label{P:LPD} Suppose $V$ satisfies \eqref{V1}--\eqref{V2} or \eqref{V3}.  Let $\{f_n\}$ be a bounded sequence in $H^1(\R^3)$. Then the following holds up to some subsequence.

There exist $J^*\in \{0,1,2,\dots,\infty\}$, non-zero profiles $\{\phi^j\}_{j=1}^{J^*}\subset H^1(\R^3)$, and space-time parameters $\{(t_n^j,x_n^j)\}_{j=1}^{J^*}\subset\R\times\R^3$ satisfying the following:

For each finite $0\leq J\leq J^*$, we can write
\begin{equation}\label{E:LPD}
f_n=\sum_{j=1}^J \phi_n^j + r_n^J,\qtq{with} \phi_n^j = e^{it_n^j H}[\phi^j(\cdot-x_n^j)] \qtq{and} r_n^J\in H^1.
\end{equation}

For each finite $0\leq J\leq J^*$ we have the following decoupling properties: 
\begin{gather}
\label{decouple1}
\lim_{n\to\infty} \bigl\{ \|H^{\frac{s}{2}} f_n\|_{L_x^2}^2- \sum_{j=1}^J \|H^{\frac{s}{2}}\phi_n^j\|_{L_x^2}^2 - \|H^{\frac{s}{2}}r_n^J\|_{L_x^2}^2 \bigr\}=0,\ s\in\{0,1\},\\
\label{decouple2}
\lim_{n\to\infty} \bigl\{\| f_n\|_{L_x^4}^4 - \sum_{j=1}^J \|\phi_n^j\|_{L_x^4}^4 - \| r_n^J\|_{L_x^4}^4\bigr\}=0.
\end{gather}

The remainder $r_n^J$ obeys
\begin{equation}\label{rnJweak}
\bigl(e^{-it_n^JH} r_n^J\bigr)(x+x_n^J) \rightharpoonup 0 \qtq{weakly in} H_x^1
\end{equation}
and vanishes in the Strichartz norm:
\begin{equation}\label{rnJ}
\lim_{J\to J^*}\limsup_{n\to\infty}\|e^{-itH}r_n^J\|_{L_{t,x}^5(\R\times\R^3)} = 0.
\end{equation}

The parameters $(t_n^j,x_n^j)$ are asymptotically orthogonal in the following sense: for any $j\neq k$,
\begin{equation}\label{orthogonal}
\lim_{n\to\infty} \bigl(|t_n^j - t_n^k| + |x_n^j - x_n^k| \bigr)= \infty.
\end{equation}

Finally, for each $j$, we may assume that either $t_n^j\to\pm\infty$ or $t_n^j\equiv 0$, and either $|x_n^j|\to\infty$ or $x_n^j\equiv 0$.
\end{proposition}

The next result states that there exist scattering solutions to \eqref{nls} corresponding to initial data below the $(\text{NLS}_0)$ threshold that are translated sufficiently far from the origin.  This will also play a role in establishing compactness for nonscattering solutions in Section~\ref{S:compact} below.  For $V$ satisfying \eqref{V3} (the inverse-square case), this result appears as Theorem~6.1 in \cite{KMVZ}; we will also need to revisit the idea of the proof in Lemma~\ref{ximpliesdeltaimpliesx} below. For potentials satisfying \eqref{V1}--\eqref{V2}, the proof is simpler and is provided below. 

\begin{proposition}[Embedding nonlinear profiles]\label{P:embedding} Let $V$ satisfy \eqref{V1}--\eqref{V2} or \eqref{V3} and suppose $\{t_n\}\subset\R$ satisfy $t_n\equiv 0$ or $t_n\to\pm\infty$, and let $\{x_n\}\subset\R^3$ satisfy $|x_n|\to\infty$. Let $\phi\in H_x^1(\R^3)$ satisfy
\begin{equation} \label{embedding-threshold}
\begin{aligned}
M(\phi)E_{0}(\phi) < M(Q)E_0(Q) \qtq{and}\|\phi\|_{L_x^2}\|\phi\|_{\dot H_x^1} < \|Q\|_{L_x^2}\|Q\|_{\dot H^1}
&\qtq{if} t_n\equiv 0, \\
\tfrac12\|\phi\|_{L_x^2}^2 \|\phi\|_{\dot H_x^1}^2 < M(Q)E_0(Q) &\qtq{if} t_n\to\pm\infty.
\end{aligned}
\end{equation}
Define
\[
\phi_n = e^{-it_nH}[\phi(\cdot-x_n)].
\]
Then for all $n$ sufficiently large, there exists a global solution $v_n$ to \eqref{nls} with $v_n(0)=\phi_n$ satisfying
\[
\| v_n\|_{S^1(\R)} \lesssim 1,
\]
with implicit constant depending on $\|\phi\|_{H_x^1}$.

Furthermore, for any $\eps>0$, there exists $N_\eps\in\mathbb{N}$ and $\psi_\eps\in C_c^\infty(\R\times\R^3)$ such that for $n\geq N_\eps$,
\begin{equation}\label{embed-cc}
\|v_n-\psi_\eps(\cdot+t_n,\cdot - x_n)\|_{X(\R\times\R^3)}  < \eta,
\end{equation}
where
\begin{align*}
& X \in \{L_{t,x}^5, L_t^5 \dot H_x^{\frac12,\frac{30}{11}}\}.
\end{align*}
\end{proposition}

\begin{proof} For $V$ satisfying \eqref{V3} (the inverse-square case), see \cite[Theorem~6.1]{KMVZ}.  For $V$ satisfying \eqref{V1}--\eqref{V2} we argue as follows.

If $t_n\equiv 0$, we let $v$ denote the solution to \eqref{nls0} with $v(0)=\phi$.  If $t_n\to\pm\infty$, we let $v$ be the solution to \eqref{nls0} that scatters to $e^{it\Delta}\phi$ as $t\to\pm\infty$.  By \eqref{embedding-threshold} and Theorem~\ref{T:sub} (with $V\equiv 0$), we obtain that in either case, $v$ is global and obeys $v\in L_t^4 L_x^6$. We now define 
\[
\tilde v_n(t,x)=v(t+t_n,x-x_n).
\]
The $\tilde v_n$ obey global $L_t^4 L_x^6$ bounds that are independent of $n$.  Furthermore, we will show that the $\tilde v_n$ are approximate solutions to \eqref{nls} that asymptotically match $\phi_n$ at $t=0$.

First, we observe that
\begin{align*}
\|\tilde v_n(0)-\phi_n\|_{L^2} & = \|v(t_n,\cdot-x_n)-e^{-it_n H}[\phi(\cdot-x_n)]\|_{L^2} \\
& = \|\{e^{it_n\Delta}-e^{-it_n H}\}[\phi(\cdot-x_n)]\|_{L^2}+o(1) \qtq{as}n\to\infty. 
\end{align*}
Now, using
\[
(i\partial_t -H)\{e^{it\Delta}-e^{-itH}\}\phi(\cdot-x_n) = -V(x)e^{it\Delta}\phi(\cdot-x_n),
\]
we may write
\[
\{e^{it\Delta}-e^{-itH}\}\phi(\cdot-x_n) = i\int_0^t e^{-i(t-s)H}Ve^{is\Delta}\phi(\cdot-x_n)\,ds.
\]
Thus, applying Strichartz and dominated convergence, we have
\begin{align*}
\|\{e^{it_n\Delta}-e^{-it_n H}\}[\phi(\cdot-x_n)]\|_{L^2} & \leq \biggl\| \int_\R e^{isH}\chi_{[0,t_n]}(s)Ve^{is\Delta}\phi(\cdot-x_n)\,ds\biggr\|_{L^2} \\
& \lesssim \|V e^{is\Delta}\phi(\cdot-x_n)\|_{L_t^2 L_x^{6/5}(\R\times\R^3)} \\
& \lesssim \|V(\cdot+x_n)\|_{L^{3/2}}\|e^{is\Delta}\phi\|_{L_t^2 L_x^6} \\
& \lesssim \|V(\cdot+x_n)\|_{L^{3/2}}\|\phi\|_{L^2}\to 0 \qtq{as}n\to\infty.
\end{align*}
Interpolating with $\dot H^1$ boundedness, we obtain
\begin{equation}\label{444data-match}
\|\tilde v_n(0)-\phi_n\|_{\dot H^{\frac12}}\to 0 \qtq{as}n\to\infty. 
\end{equation}

Next, we observe that 
\[
(i\partial_t + \Delta - V)\tilde v_n + |\tilde v_n|^2 \tilde v_n = -V \tilde v_n
\]
and that
\begin{equation}\label{444error}
\|V\tilde v_n\|_{L_t^4 L_x^{6/5}(\R\times\R^3)} \lesssim \|V(\cdot+x_n)\|_{L^{\frac32}}\|v\|_{L_t^4 L_x^6} \to 0 \qtq{as}n\to\infty.
\end{equation}

In light of \eqref{444data-match} and \eqref{444error}, we may apply the stability result to deduce the existence of true solutions $v_n$ to \eqref{nls} satisfying $v_n(0)=0$ and obeying the global space-time bounds.  One can then obtain the approximation in \eqref{embed-cc} by utilizing \eqref{E:stab}. \end{proof}

\subsection{Variational analysis}

The scattering threshold is related to the sharp constant $C_V$ for the following Gagliardo--Nirenberg inequality:
\begin{equation}\label{sharp-GN}
\|f\|_{L^4}^4 \leq C_V \|f\|_{L^2} \|f\|_{\dot H_V^1}^3
\end{equation}
For $V=0$, equality is attained by the ground state $Q$. On the other hand, the following result is proven in \cite[Theorem~3.1]{KMVZ} and \cite[Proposition~1.1]{Hong}.

\begin{lemma}\label{no-optimizers} For $V$ satisfying \eqref{V1}--\eqref{V2} or \eqref{V3}, the sharp constant in \eqref{sharp-GN} is given by $C_V=C_0$; however, equality is never attained. 
\end{lemma}

\begin{corollary}\label{C:global} If $u_0\in H^1(\R^3)$ satisfies 
\begin{equation}\label{threshold0}
M(u_0)=M(Q),\quad E_V(u_0)=E_0(Q),\qtq{and} \|u_0\|_{\dot H_V^1}<\|Q\|_{\dot H^1},
\end{equation}
then the corresponding solution to \eqref{nls} is global-in-time and uniformly bounded in $H^1$. 
\end{corollary}

\begin{proof} Recalling the local theory for \eqref{nls} and the fact that $M(u)=M(Q)$, it suffices to show that the condition
\[
\|u(t)\|_{\dot H_V^1}<\|Q\|_{\dot H^1}
\]
persists throughout the lifespan $I$ of $u$.  In fact, if $\|u(t)\|_{\dot H_V^1}=\|Q\|_{\dot H^1}$ for some $t\in I$, then the condition $E_V(u)=E_0(Q)$ yields $\|u(t)\|_{L^4}=\|Q\|_{L^4}$. In particular (recalling $C_V=C_0$), we find that 
\[
\|Q\|_{L^4}^4 = C_0 \|Q|\|_{L^2} \|Q\|_{\dot H^1}^3 \implies \| u(t)\|_{L^4}^4 = C_a \|u(t)\|_{L^2}\|u(t)\|_{\dot H_V^1}^3,
\]
which contradicts Lemma~\ref{no-optimizers}.\end{proof}

\subsection{Virial identities}\label{S:virial}

We let $\phi$ be a real-valued, radial function satisfying
\[
\phi(x) = \begin{cases} |x|^2 & |x|\leq 1, \\ 4 & |x|>3\end{cases} \qtq{and} |\partial^\alpha \phi(x)| \lesssim |x|^{2-|\alpha|}.
\]
We further impose $\partial_r \phi \geq 0$, where $\partial_r$ denotes the radial derivative.  

Given $R\geq 1$, we define 
\[
w_R(x) = R^2 \phi(\tfrac{x}{R}) \qtq{and} w_\infty(x) = |x|^2,
\]
and for $R\in[1,\infty]$ we define the functional
\[
P_R[u] = 2\Im \int \bar u \nabla u \cdot \nabla w_R \,dx. 
\]

We write the standard (truncated) virial identity as follows. 
\begin{lemma}[Standard virial identity]\label{L:virial} Fix $R\in[1,\infty]$.  If $u(t)$ solves \eqref{nls}, then 
\begin{equation}\label{virial-id}
\tfrac{d}{dt}P_R[u]  = F_R^V[u(t)],
\end{equation}
where
\begin{align*}
F_R^V[u]  & := \int (-\Delta\Delta w_R)|u|^2 + 4\Re \bar u_j u_k \partial_{jk}[w_R] - 2|u|^2\nabla V\cdot\nabla w_R - |u|^4 \Delta w_R\,dx \\
& = F_R^0[u] - 2\int |u|^2 \nabla V\cdot\nabla w_R\,dx. 
\end{align*}
\end{lemma}

When $R=\infty$, we obtain the standard virial quantity
\begin{equation}\label{usual-V}
F_\infty^V[u] = 8\int_{\R^3} |\nabla u|^2 - \tfrac34 |u|^4\,dx - 4\int_{\R^3} x\cdot\nabla V(x)|u|^2\,dx. 
\end{equation}
This quantity controls the quantity $\delta(t)$ that plays an essential role in Section~\ref{S:impossible} below. In standard applications of the virial identity, one adds and subtracts $F_\infty^V[u]$ to the right-hand side of \eqref{virial-id}, treating $F_R^V[u]-F_\infty^V[u]$ as an error term.

In addition to this standard application, we will need to make use of a `modulated' virial identity, which incorporates the possibility that the solution to \eqref{nls} approaches the orbit of the standard NLS ground state, that is, the set
\[
\{e^{i\theta}Q(\cdot-y):\theta\in\R,\ y\in\R^3\}.
\]

The following lemma will allow us to `add zero' in the virial identity in a convenient form.

\begin{lemma}\label{L:add-zero} For any $R\in[1,\infty]$, $\theta\in\R$, and $y\in\R^3$, we have
\[
F_R^0[e^{i\theta}Q(\cdot-y)] = 0. 
\]
\end{lemma}

\begin{proof} Fix $R\in[1,\infty]$, $\theta\in\R$, and $y\in\R^3$.  As $Q$ is real-valued, 
\[
P_R[e^{it}e^{i\theta}Q(\cdot-y)]=0 \qtq{for all}t\in\R.
\]
As $e^{it}e^{i\theta}Q(\cdot-y)$ solves $(\text{NLS}_0)$, Lemma~\ref{L:virial} implies
\[
F_R^0[e^{it}e^{i\theta}Q(\cdot-y)]=0 \qtq{for all}t\in\R. 
\]
Evaluating at $t=0$ yields the desired result. \end{proof}

Using the identities above, we obtain the following corollary.

\begin{corollary}\label{C:virial} Let $u$ be a solution to \eqref{nls} on an interval $I$.  For any $R\in[1,\infty)$, $\chi:I\to\R$, $\theta:I\to\R$, and $y:I\to\R^3$,  we have
\begin{align}
\tfrac{d}{dt}P_R[u] & = F_\infty^0[u(t)] \nonumber \\
& \quad + F_R^V[u(t)]-F_\infty^0[u(t)] \label{the-error-term} \\
& \quad - \chi(t)\bigl\{F_R^0[e^{i\theta(t)}Q(\cdot-y(t))]-F_\infty^0[e^{i\theta(t)}Q(\cdot-y(t))]\bigr\} \label{add-zero}
\end{align}
for $t\in I$. 
\end{corollary}

\begin{proof}  By Lemma~\ref{L:add-zero}, we have that \eqref{add-zero}$=0$ for all $t\in I$.  Thus this identity is simply a rewriting of \eqref{virial-id}. 
\end{proof}

The idea behind Corollary~\ref{C:virial} is that the terms in \eqref{the-error-term} and \eqref{add-zero} can be read in two ways. First, with $\chi(t)=0$ we recover the standard virial error term in \eqref{the-error-term}.  On the other hand, if $u$ is known to be close to the orbit of $Q$, then we can exploit this fact by choosing $\chi(t)=1$, as the combination of \eqref{the-error-term} and \eqref{add-zero} basically takes the form
\[
G(u(t))-G(e^{i\theta(t)}Q(\cdot-y(t)))
\]
for some continuous functional $G$. 
\section{Compactness for nonscattering solutions}\label{S:compact}

Our first main step towards the proof of Theorem~\ref{T} is the following proposition, which shows that nonscattering threshold solutions must be `compact' modulo some time-dependent spatial center.    The proof follows the same strategy as the proof of \cite[Proposition~4.1]{DLR}, and relies primarily on results from our previous work \cite{KMVZ} as well as the work of Hong \cite{Hong}.  Accordingly, our presentation will be somewhat abbreviated.

\subsection{Compactness for nonscattering solutions}

\begin{proposition}\label{P:compact}  Suppose Theorem~\ref{T} fails for some potential $V$ satisfying \eqref{V1}--\eqref{V2} or \eqref{V3}.  Then there exists a forward-global solution $u$ to \eqref{nls} obeying
\begin{align}
&M(u_0)=M(Q),\quad E_V(u_0)=E_0(Q),\qtq{and} \|u_0\|_{\dot H_V^1}<\|Q\|_{\dot H^1},\label{threshold} \\
&\|u\|_{L_{t,x}^5([0,\infty)\times\R^3)}=\infty,\label{blowup}
\end{align}
and there exists $x:[0,\infty)\to\R^3$ such that
\begin{equation}\label{compact-orbit}
\{u(t,\cdot+x(t)):t\in[0,\infty)\}\qtq{is pre-compact in}H^1.
\end{equation}
\end{proposition}

\begin{proof}[Proof of Proposition~\ref{P:compact}]  Our first goal is to show that if Theorem~\ref{T} fails, then we may find a forward-global, $H^1$ bounded solution $u$ obeying \eqref{threshold} and \eqref{blowup}. 

\begin{proof}[Case 1.] First suppose $V$ satisfies \eqref{V3}, i.e. $V=a|x|^{-2}$ for some $a>0$. In this case, the equation \eqref{nls} enjoys the following scaling symmetry: 
\begin{equation}\label{scaling}
u(t,x)\mapsto \lambda u(\lambda^2 t, \lambda x).
\end{equation}

We now suppose that Theorem~\ref{T} fails and seek a solution as above.  In particular, we suppose that there exists a solution $v$ obeying \eqref{SUBTH} that fails to scatter forward in time.  If we then define $u_0(x)=\lambda v_0(\lambda x)$, where $\lambda=\tfrac{M(Q)}{M(v_0)}$, then the corresponding solution $u$ obeys \eqref{threshold} and is given by
\begin{equation}\label{scaling2}
u(t,x)=\lambda v(\lambda^2 t,\lambda x).
\end{equation}
By Corollary~\ref{C:global} we have that $u$ is forward-global and $H^1$-bounded, while \eqref{scaling2} guarantees \eqref{blowup}.\end{proof}

\begin{proof}[Case 2.]  Next suppose $V$ satisfies \eqref{V1}--\eqref{V2}.  To pass from \eqref{SUBTH} to \eqref{threshold}, we use the following. 

\begin{lemma}\label{L:change-hypothesis} Suppose that Theorem~\ref{T} holds for  potentials satisfying \eqref{V1}--\eqref{V2} with the hypothesis \eqref{SUBTH} replaced by \eqref{threshold}.  Then Theorem~\ref{T} holds with the original hypothesis \eqref{SUBTH}.
\end{lemma} 

\begin{proof}  We suppose Theorem~\ref{T} holds with the hypothesis \eqref{threshold}.  We then take $V$ satisfying \eqref{V1}--\eqref{V2} and $u_0\in H^1$ obeying \eqref{SUBTH}, and we let $u$ denote the corresponding solution to \eqref{nls}.  Our task is to deduce global space-time bounds for $u$. 

To this end, we let $v_0(x)=u_0(\lambda x)$, where $\lambda=\tfrac{M(Q)}{M(u_0)}$, and define
\begin{equation}\label{vu-scaling}
v(t,x) = \lambda u(\lambda^2 t,\lambda x). 
\end{equation}
Writing $V_\lambda$ for the rescaled potential
\[
V_\lambda(x) = \lambda^2 V(\lambda x),
\]
it follows from \eqref{SUBTH} that $v_0$ satisfies \eqref{threshold} for the potential $V_\lambda$, i.e.
\[
M(v_0) = M(Q),\quad E_{V_{\lambda}}(v_0)=E_0(Q),\qtq{and} \|v_0\|_{\dot H^1_{V_\lambda}}<\|Q\|_{\dot H^1},
\]
and solves the equation $(\text{NLS}_{V_\lambda})$, i.e.
\[
i\partial_t v + \Delta v - V_\lambda v = -|v|^2 v. 
\] 
Observing that $V_\lambda$ still satisfies \eqref{V1}--\eqref{V2}, it follows by hypothesis that $v$ obeys global $L_{t,x}^5$ space-time bounds.  In light of \eqref{vu-scaling}, we deduce that global space-time bounds hold for $u$ as well. \end{proof}

In light of Lemma~\ref{L:change-hypothesis}, we see that if Theorem~\ref{T} fails (with the original hypothesis \eqref{SUBTH}), then it fails with the hypothesis \eqref{threshold}.  Therefore, recalling Corollary~\ref{C:global}, we may find a forward-global, $H^1$ bounded solution $u$ obeying \eqref{threshold} and \eqref{blowup}, as desired.\end{proof}

We now let $u$ be a forward-global solution to \eqref{nls} obeying \eqref{threshold} and \eqref{blowup}.  Observe that by the proof of Corollary~\ref{C:global}, we also have that
\begin{equation}\label{still-below}
\|u(t)\|_{\dot H_V^1} < \|Q\|_{L^2}\qtq{for all}t\geq 0.
\end{equation}
To complete the proof, it suffices to show that for any $\tau_n\to\infty$, there exists a subsequence in $n$ and a sequence $x_n\in\R^3$ so that 
\[
u(\tau_n,x+x_n)\qtq{converges strongly in}H^1.
\]
(The existence of the function $x(t)$ as a consequence of this is discussed below in Section~\ref{S:xt}.)

We now apply the linear profile decomposition (Proposition~\ref{P:LPD}) to the bounded sequence $u(\tau_n)$ to obtain 
\begin{equation}\label{equ:LPD-1}
u_n:=u(\tau_n)=\sum_{j=1}^J \phi_n^j + r_n^J
\end{equation}
along a subsequence, with all of the properties stated in that proposition.  We single out three possibilities, namely: vanishing ($J^*=0$), compactness ($J^*=1)$, or dichotomy ($J^*\geq 2$). 

If $J^*=0$, then we obtain 
\[
\|e^{-itH}u(\tau_n)\|_{L_{t,x}^5([0,\infty)\times\R^3)} \to 0 \qtq{as}n\to\infty. 
\]
Using this together with stability (i.e. Proposition~\ref{P:stab}), we readily deduce 
\[
\|u(t+\tau_n)\|_{L_{t,x}^5((0,\infty)\times\R^3)} = \|u\|_{L_{t,x}^5((\tau_n,\infty)\times\R^3} \lesssim 1
\]
for all large $n$, which contradicts \eqref{blowup}.  Thus vanishing cannot occur. 

Next, we rule out dichotomy (i.e. $J^*\geq 2$).  If $J^*\geq 2$, then we can firstly utilize the decoupling of mass and energy (see \eqref{decouple1} and \eqref{decouple2}).  In particular, recalling \eqref{threshold} and \eqref{still-below}, we have the following for any $0\leq J\leq J^*$
\begin{align}\label{equ:masdec}
\lim_{n\to\infty}\Big\{\sum_{j=1}^J  M(\phi_n^j)+M(r_n^J)\Big\}&=\lim_{n\to\infty}M(u_n)=M(Q)\\\label{equ:enedec}
\lim_{n\to \infty}\Big\{\sum_{j=1}^J E_V(\phi_n^j)+E_V(r_n^J)\Big\}&=\lim_{n\to\infty}E_V(u_n)=E(Q), \\
\limsup_{n\to\infty} \Big\{ \sum_{j=1}^J \| \phi_n^j\|_{\dot H_V^1}^2 + \|r_n^J\|_{\dot H_V^1}^2\bigr\}& =\limsup_{n\to\infty}\|u_n\|_{\dot H_V^1}^2  \leq \|Q\|_{\dot H^1}^2. \label{equ:stillbelow}
\end{align}
Using this, we can show that each $\phi_n^j$ satisfies the sub-threshold hypotheses appearing in Theorem~\ref{T:sub} below; in fact, there exists some $\delta>0$ so that 
\begin{equation}\label{sub-threshold-quant}
M(\phi_n^j)E_V(\phi_n^j)<M(Q)E_0(Q)-\delta
\end{equation}
for all $n$ large and $1\leq j \leq J$. The proof of this hinges on the fact that 
\begin{equation}\label{positive-energy}
\liminf_{n\to\infty} E_V(\phi_n^j)>0 
\end{equation}
for $1\leq j\leq J,$ which we can see as follows: First, if $t_n^j\to\pm\infty$, then the nonlinear part of the energy tends to zero as $n\to\infty$ (see e.g. \cite[Corollary~2.13]{KMVZ} and \cite[Corollary~5.2]{Hong}), which yields \eqref{positive-energy}.  Thus we may assume $t_n^j\equiv 0$.  Now, whether $|x_n^j|\to\infty$ or $x_n^j\equiv 0$, we may obtain
\[
\liminf_{n\to\infty} E_V(\phi_n^j)\geq \tfrac12\|\phi^j\|_{\dot H^1}^2 - \tfrac14\|\phi^j\|_{L^4}^4= E_0(\phi^j). 
\] 
As \eqref{equ:masdec} and \eqref{equ:stillbelow} imply $\|\phi^j\|_{L^2}\leq\|Q\|_{L^2}$ and $\|\phi^j\|_{\dot H^1}\leq \|Q\|_{\dot H^1}$, we may use the sharp Gagliardo--Nirenberg inequality \eqref{sharp-GN} (with $V=0$) to deduce
\[
E_0(\phi^j) \geq\tfrac12\|\phi^j\|_{\dot H^1}^2\{1-\tfrac12 C_0\|Q\|_{L^2}\|Q\|_{\dot H^1}\}\geq \tfrac16\|\phi^j\|_{\dot H^1}^2 
\]
(see e.g. \cite[(3.5)]{KMVZ}). 

To reach a contradiction under the dichotomy assumption, we will use the linear profiles $\phi_n^j$ to build approximate solutions to \eqref{nls}.

First consider the case $x_n^j\equiv 0$. If $t_n^j\equiv 0$,  then we appeal to Theorem~\ref{T:sub} and \eqref{sub-threshold-quant} to obtain a solution $v^j$ to \eqref{nls} satisfying $v^j(0)=\phi^j$ and obeying global space-time bounds (depending only on $\delta>0$ in \eqref{sub-threshold-quant} above). If instead $t_n^j\to\pm\infty$, then we let $v^j$ be the solution to \eqref{nls} that scatters to $e^{-itH}\phi^j$ as $t\to\pm\infty$.  As we can again verify \eqref{sub-threshold-quant}, we obtain global space-time bounds for the solution in this case as well. In either case, we define
\[
v_n^j(t,x)=v^j(t+t_n^j,x).
\]

Next, if $|x_n^j|\to\infty$, we appeal to Proposition~\ref{P:embedding} to obtain a solution $v_n^j$ to \eqref{nls} satisfying $v_n^j(0)=\phi_n^j$ and obeying global space-time bounds. 

We now define approximate solutions to \eqref{nls} by
\[
u_n^J(t,x) = \sum_{j=1}^J v_n^j(t,x)+ e^{-itH}r_n^J
\]
and immediately observe that by construction,
\begin{equation}\label{unJ-0-match}
\lim_{n\to\infty} \|u_n^J(0)- u_n\|_{H^1}=0 \qtq{for each}J.
\end{equation}
Our next goal is to show the following:
\begin{align}
&\limsup_{J\to J^*}\limsup_{n\to\infty}\bigl\{ \|u_n^J(0)\|_{H^1}+\|u_n^J\|_{L_{t,x}^5}\bigr\}\lesssim 1, \label{unJ-bounds}\\
&\limsup_{J\to J^*}\limsup_{n\to\infty} \bigl\{ |\nabla|^{\frac12}[(i\partial_t - H)u_n^J + |u_n^J|^2 u_n^J]\|_{L_t^{\frac{5}{3}}L_x^{\frac{30}{23}}}\bigr\} = 0,\label{unJ-error}
\end{align}
where space-time norms are taken over $\R\times\R^3$.  Once we have established \eqref{unJ-bounds} and \eqref{unJ-error}, the stability result Proposition~\ref{P:stab} implies that $u\in L_{t,x}^5(\R\times\R^3)$, contradicting \eqref{blowup} and hence ruling out the possibility of dichotomy.

Before turning to the proof of \eqref{unJ-bounds} and \eqref{unJ-error}, we observe that by the orthogonality of the parameters \eqref{orthogonal} and approximation by functions in $C_c^\infty(\R\times\R^3)$ (see \eqref{embed-cc} for the case $|x_n^j|\to\infty$), we may obtain the following:  for any $j\neq k$,
\begin{equation}\label{orthogonaljk}
\|v_n^j v_n^k\|_{L_{t,x}^{\frac52}}+ \|v_n^j |\nabla|^{\frac12}v_n^k\|_{L_t^{\frac52}L_x^{\frac{30}{17}}} \to 0 
\end{equation}
as $n\to\infty$.

\begin{proof}[Proof of \eqref{unJ-bounds}] First, the $H^1$ bound follows from \eqref{unJ-0-match}.  Next, using the $H^1$ decoupling and following the argument used to obtain \cite[(7.14)]{KMVZ}, we can show that
\[
\limsup_{J\to J^*}\limsup_{n\to\infty}\sum_{j=1}^J \|v_n^j\|_{S_V^1}^2\lesssim 1.
\]
Together with \eqref{orthogonaljk} and Strichartz, this implies the desired $L_{t,x}^5$-bounds.  \end{proof}

Let us also observe for later use that similar arguments yield the bounds 
\begin{equation}\label{more-unJ-bounds}
\limsup_{J\to J^*}\limsup_{n\to\infty} \| |\nabla|^{\frac12} u_n^J\|_{L_t^5 L_x^{\frac{30}{11}}}\lesssim 1. 
\end{equation}
\begin{proof}[Proof of \eqref{unJ-error}] We write $F(z)=-|z|^2 z$, so that we are led to estimate the following two terms:
\begin{align}
&\biggl\| |\nabla|^{\frac12}\biggl[\sum_{j=1}^J F(v_n^j)-F\bigl(\sum_{j=1}^J v_n^j\bigr)\biggr]\biggr\|_{L_t^{\frac53}L_x^{\frac{30}{23}}},\label{enJ1} \\
&\|\,|\nabla|^{\frac12}[F(u_n^J-e^{-itH}r_n^J)-F(u_n^J)]\|_{L_t^{\frac53}L_x^{\frac{30}{23}}}.\label{enJ2}
\end{align}

We begin with \eqref{enJ1} and observe that the expression
\[
\sum_{j=1}^J F(v_n^j)-F\bigl(\sum_{j=1}^J v_n^j\bigr) 
\]
may be written as a finite linear combination of terms of the form $v_n^j v_n^k v_n^\ell$ (up to complex conjugates), where not all of $j,k,\ell$ are equal.  The coefficients and total number of terms are $J$-dependent; however, this dependence becomes irrelevant once we establish 
\begin{equation}\label{paraproduct}
\lim_{n\to\infty} \| |\nabla|^{\frac12}[v_n^j v_n^k v_n^\ell]\|_{L_t^{\frac53}L_x^{\frac{30}{23}}}=0 \qtq{for all such}j,k,\ell.
\end{equation}
To prove \eqref{paraproduct}, we argue as in \cite{KM-cubic}.  That is, supposing $j\neq k$ (say), we first use the fractional product rule to bound this quantity by
\[
\|v_n^j v_n^k\|_{L_{t,x}^{\frac52}}\| |\nabla|^{\frac12}v_n^\ell\|_{L_t^5 L_x^{\frac{30}{11}}} + \| |\nabla|^{\frac12}[v_n^j v_n^k]\|_{L_{t,x}^{\frac52}L_x^{\frac{30}{17}}}\|v_n^\ell\|_{L_{t,x}^5}
\]
and observe that the first quantity is $o(1)$ as $n\to\infty$ by \eqref{orthogonaljk}. For the second quantity, we use the paraproduct estimate appearing in \cite[(3.6)]{KM-cubic} to write
\begin{align*}
\| & |\nabla|^{\frac12}[v_n^j v_n^k]\|_{L_t^{\frac52}L_x^{\frac{30}{17}}} \\
& \leq \|v_n^j |\nabla|^{\frac12}v_n^k\|_{L_t^{\frac52}L_x^{\frac{30}{17}}} + \| v_n^k |\nabla|^{\frac12}v_n^j\|_{L_t^{\frac52}L_x^{\frac{30}{17}}} + \|A (v_n^j) B(|\nabla|^{\frac12}v_n^k)\|_{L_t^{\frac52}L_x^{\frac{30}{17}}},
\end{align*}
where $A$ and $B$ are bounded sublinear operators that commute with translation. In particular, the first two terms are $o(1)$ as $n\to\infty$ by \eqref{orthogonaljk}, while the final term is $o(1)$ by the same arguments used to prove \eqref{orthogonaljk}. 

Finally, we turn to \eqref{enJ2}. For this term, we observe that
\[
F(u_n^J-e^{-itH}r_n^J)-F(u_n^J)
\]
is a sum of terms of the form $[e^{-itH}r_n^J]\star\,\star$ where $\star\in\{e^{-itH}r_n^J,u_n^J\}$ (up to complex conjugates). Thus, by the fractional product rule, \eqref{rnJ}, Strichartz, \eqref{unJ-bounds}, and \eqref{more-unJ-bounds}, we can obtain
\begin{align*}
\eqref{enJ2} & \lesssim \| e^{-itH} r_n^J\|_{L_{t,x}^5}\bigl(\|r_n^J\|_{\dot H^{\frac12}}+\|u_n^J\|_{L_{t,x}^5}\bigr)\bigl(\|r_n^J\|_{\dot H^{\frac12}} + \||\nabla|^{\frac12} u_n^J\|_{L_t^5 L_x^{\frac{30}{11}}}\bigr) \\
& \lesssim \|e^{-itH}r_n^J\|_{L_{t,x}^5}\to 0 \qtq{as}n\to\infty\qtq{and}J\to J^*. 
\end{align*} 
\end{proof}

Having ruled out vanishing and dichotomy, we are left with the conclusion that $J^*=1$ (`compactness'). In particular, our decomposition \eqref{equ:LPD-1} reduces to
\[
u_n=u(\tau_n) = \phi_n + r_n = e^{it_n H}[\phi(\cdot-x_n)]+ r_n.
\]
Using mass and energy decoupling (and positivity of the energy), we can guarantee that $r_n\to 0$ strongly in $H^1$. Indeed, the profile decomposition already converges weakly to zero in $H^1$, so if strong convergence fails then we may obtain that $\phi_n$ obeys the sub-threshold hypothesis \eqref{sub-threshold-quant}.  Then the same argument used to prevent dichotomy implies scattering for $u$ via the stability theory, contradicting \eqref{blowup}.  All that remains is to preclude the possibility that $t_n\to\pm\infty$.

For this, we suppose towards a contradiction that $t_n\to\infty$ (say). We then observe that
\[
\|e^{-itH}\phi_n\|_{L_{t,x}^5([0,\infty)\times\R^3)}  = \|e^{-itH}[\phi(\cdot-x_n)]\|_{L_{t,x}^5([t_n,\infty)\times\R^3)}\to 0
\]
as $n\to\infty$ by the monotone convergence theorem. Using this, we can verify that for all $n$ large, $e^{-itH}\phi_n$ defines a good approximate solution to \eqref{nls} that obeys $L_{t,x}^5$ bounds on $[0,\infty)$ and asymptotically matches $u_n$ in $H^1$. Applying Proposition~\ref{P:stab} once more, we therefore reach a contradiction to \eqref{blowup}. \end{proof}


\subsection{Aside: defining $x(t)$}\label{S:xt} The proof of Proposition~\ref{P:compact} shows that for any sequence $\{t_n\}\subset[0,\infty)$, there exist $x_n\in\R^3$ and nonzero $\psi\in H^1$ such that
\begin{equation}\label{apply-P:compact}
u(t_n,x+x_n)\to \psi \qtq{strongly in}H^1
\end{equation}
along a subsequence.  To deduce the existence of a function $x:[0,\infty)\to\R^3$ such that \eqref{compact-orbit} holds, one approach is as follows.  We first claim that there exists $c>0$ such that
\[
\sup_{x_0\in\R^3} \|u(t)\|_{L^2(B_1(x_0))}> c \qtq{for all}t\in[0,\infty). 
\]
If not, then we may find a sequence $t_n$ along which this supremum tends to zero.  Passing to a subsequence and choosing $x_n$ and $\psi$ as in \eqref{apply-P:compact}, we deduce that for any $x_0\in\R^3$,
\[
\|\psi\|_{L^2(B_1(x_0))}\leq \|\psi-u(t_n,\cdot+x_n)\|_{L^2}+\|u(t_n)\|_{L^2(B_1(x_0+x_n))}\to 0 \qtq{as}n\to\infty.
\]
This yields the contradiction $\psi\equiv 0$.  Thus we may define a function $x:[0,\infty)\times\R^3$ such that
\begin{equation}\label{x(t)-lower-bound-on-unit-ball}
\|u(t)\|_{L^2(B_1(x(t)))}\geq c \qtq{for all}t\in[0,\infty). 
\end{equation}
It remains to verify \eqref{compact-orbit}.  Again, we take an arbitrary sequence $\{t_n\}$ and choose $x_n$ and $\psi$ as in \eqref{apply-P:compact} to deduce
\[
u(t_n,x+x(t_n))=\psi(x+x(t_n)-x_n)+o(1) \qtq{in}H^1 
\]
along some subsequence. Thus either $u(t_n,x+x(t_n))$ converges strongly in $H^1$ (as desired) or converges weakly to zero. However, weak convergence to zero is incompatible with \eqref{x(t)-lower-bound-on-unit-ball}, and so we conclude that \eqref{compact-orbit} holds.

\section{Impossibility of compact solutions}\label{S:impossible}
Throughout this section, we suppose $V$ is a potential satisfying \eqref{V1}--\eqref{V2} or \eqref{V3} and that that $u(t)$ is a solution to \eqref{nls} as in Proposition~\ref{P:compact}.  In particular (by Corollary~\ref{C:global}), $u$ is global and uniformly bounded in $H^1$, but $u(t)$ does not scatter forward in time, and there exists $x_0:[0,\infty)\to\R^3$ such that 
\begin{equation}\label{u-is-compact-x0}
\{u(t,\cdot+x_0(t)):t\in[0,\infty)\}\qtq{is pre-compact in}H^1.
\end{equation}
We will derive a contradiction by showing that the spatial center can be neither bounded nor unbounded.  Roughly speaking, as long as the spatial center is bounded, the standard virial identity serves to push the solution away from the origin.  On the other hand, if the solution moves far from the origin, then we can show that $u(t)$ approaches the orbit of $Q$, that is, the set 
\[
\{e^{i\theta}Q(\cdot-y):\theta\in\R,y\in\R^3\}.
\]
In this case, we can apply a `modulated virial' argument that ultimately limits the motion of the solution.

The arguments concerning the spatial center are mediated by a quantity denoted by $\delta(t)$, which (as we will see) provides a measure of the distance between $u(t)$ and the orbit of $Q$.  In particular, we define the functional
\[
\delta(v) := \int |\nabla Q|^2 \,dx - \int |\nabla v|^2+V(x)|v|^2 \,dx 
\]
and write $\delta(t):=\delta(u(t))$.  Observe that as in the proof of Corollary~\ref{C:global}, we have
\begin{equation}\label{delta-is-always-positive}
\delta(t)>0 \qtq{for all}t\in[0,\infty),
\end{equation}
although without any quantitative lower bound.

The relationship between the spatial center of $u$ and the quantity $\delta(t)$ is (roughly) that the spatial center becomes large if and only if $\delta(t)$ becomes small.  Therefore, it will be essential to understand the regime in which $\delta(t)$ is small.  Given a parameter $\delta_0>0$ (which will be taken sufficiently small at several points in the argument), we define the set
\[
I_0 = \{t\in[0,\infty):\delta(t)<\delta_0\},
\]
which (by continuity of the flow in $H^1$) is relatively open in $[0,\infty)$.  We then have the following description of $u(t)$ on the set $I_0$.

\begin{proposition}[Modulation]\label{P:modulation} The following holds for any $\delta_0>0$ sufficiently small: There exist $\theta:I_0\to\R$ and $y:I_0\to\R^3$ so that $u(t)$ admits the decomposition
\begin{equation}\label{modulation-identity}
u(t) = e^{i\theta(t)}[Q(x-y(t))+g(t)] \qtq{for}t\in I_0,
\end{equation}
with
\begin{equation}\label{modulation-bounds}
\frac{e^{-2|y(t)|}}{|y(t)|^2}+ |\dot y(t)| + \biggl[\int V(x)|u(t,x)|^2\,dx\biggr]^{\frac12} \lesssim \delta(t) \sim \|g(t)\|_{H^1}
\end{equation}
for all $t\in I_0.$
\end{proposition}

While this proposition is crucial for the arguments below, the proof of this result is essentially independent from the rest of the analysis, and so we postpone it until Section~\ref{S:modulation} below.  For now, let us deal with the small annoyance that we presently have two spatial centers parametrizing our solution, namely, $x_0(t)$ for $t\in[0,\infty)$ and $y(t)$ for $t\in I_0$. 

\begin{lemma}\label{L:xvsy} Suppose $\delta_0=\delta_0(u)$ is sufficiently small.  Then there exists $C>0$ such that for all $t\in I_0$, $|x_0(t)-y(t)|<C$. 
\end{lemma}

\begin{proof} Suppose towards a contradiction that $|x_0(t_n)-y(t_n)|\to\infty$ for some sequence of $t_n\in I_0$.  By compactness (see e.g. \eqref{x(t)-lower-bound-on-unit-ball} above) and Proposition~\ref{P:modulation}, there exists $c=c(u)>0$ so that
\begin{align*}
c& <\int_{|x-x_0(t_n)|\leq 1} |u(t_n,x)|^2 \,dx \\
&  \lesssim \int_{|x-[x_0(t_n)-y(t_n)]|\leq 1} |Q(x)|^2\,dx` + \int_{|x-x_0(t_n)|\leq 1} |g(t_n,x)|^2\,dx \\
& \lesssim \int |g(t_n,x)|^2\,dx + o(1) \qtq{as}n\to\infty. 
\end{align*}
Thus, by \eqref{modulation-bounds} in Proposition~\ref{P:modulation}, we have
\[
\tfrac12 c<\|g(t_n)\|_{L^2}^2 \lesssim [\delta(t_n)]^2 \lesssim \delta_0^2
\]
for all large $n$, and hence we deduce a contradiction provided $\delta_0=\delta_0(c)$ is small enough. \end{proof}

We may therefore define a new spatial center $x:[0,\infty)\to\R^3$ by setting
\begin{equation}\label{def:x(t)}
x(t) = \begin{cases} x_0(t) & t\in[0,\infty)\backslash I_0, \\ y(t) & t\in I_0.\end{cases}
\end{equation}
In light of Lemma~\ref{L:xvsy}, we maintain the compactness condition \eqref{u-is-compact-x0}.  That is,
\begin{equation}\label{u-is-compact}
\{u(t,\cdot+x(t)):t\in[0,\infty)\} \qtq{is pre-compact in}H^1. 
\end{equation} 
In Sections~\ref{BtoUB} and \ref{UBtoB}, we will obtain our contradiction by showing that $x(t)$ may be neither bounded nor unbounded.  Before this, however, let us collect a few lemmas that elucidate the relationships between the various quantities that have been introduced above.

\subsection{Dramatis personae}\label{S:DP}  In this section, we seek to clarify the relationships between the main characters appearing in the analysis below, namely, the spatial center $x(t)$; the `distance' $\delta(t)$; the potential part of the energy
\[
\int_{\R^3} V(x)|u(t,x)|^2\,dx;
\]
and the standard virial quantity
\[
F_\infty^0[u(t)]=\int_{\R^3} |\nabla u(t,x)|^2 - \tfrac34|u(t,x)|^4\,dx. 
\]

We begin with the following lemma.

\begin{lemma}\label{xt-vs-potential} For any sequence $\{t_n\}\subset[0,\infty)$, we have
\[
|x(t_n)|\to\infty \qtq{if and only if} \int V(x)|u(t_n,x)|^2\,dx \to 0. 
\]
\end{lemma}

\begin{proof} First suppose $|x(t_n)|\to\infty$ and let $\eps>0$. We first claim that using compactness (i.e. \eqref{u-is-compact}), we may find $C_\eps$ sufficiently large that
\[
\int_{|x-x(t_n)|>C_\eps}V(x)|u(t_n,x)|^2\,dx < \eps \qtq{for all}n. 
\]
In fact, this follows from the continuity of the embedding $H^1\hookrightarrow L^2(\sqrt{V}\,dx)$.  In particular, if $V=a|x|^{-2}$, this follows from Hardy's inequality, while if $V$ satisfies \eqref{V1}--\eqref{V2}, this follows from the fact that $V\in L^{\frac32}$, H\"older's inequality, and Sobolev embedding.  On the other hand, the integral over the region $|x-x(t_n)|\leq C_\eps$ tends to zero by the dominated convergence theorem.
%

Next we suppose that
\begin{equation}\label{Vutnto0}
\int V(x)|u(t_n,x)|^2\,dx \to 0
\end{equation}
but $x(t_n)$ converges along some subsequence. It follows from compactness (i.e. \eqref{u-is-compact}) that $u(t_n)$ converges strongly in $H^1$ along a further sequence to some nonzero limit $\phi\in H^1$.  Using \eqref{Vutnto0}, however, we deduce the contradiction
\[
\int V(x) |\phi(x)|^2\,dx = 0.
\]
\end{proof}

We next establish the following lemma connecting $x(t)$ and $\delta(t)$.  For this result, we rely on the fact that the standard NLS is essentially the `limit' of \eqref{nls} in the regime $|x|\to\infty$.  In particular, we rely on the results of \cite{DR} concerning threshold behaviors for the cubic NLS and the arguments of \cite{Hong, KMVZ} concerning the approximation of \eqref{nls} by ($\text{NLS}_0$) in the regime $|x|\to\infty$. 

\begin{lemma}\label{ximpliesdeltaimpliesx} Let $t_n\to\infty$.  Then
\[
|x(t_n)|\to\infty \qtq{if and only if} \delta(t_n)\to 0. 
\]
\end{lemma}

\begin{proof} If $\delta(t_n)\to 0$, then by Proposition~\ref{P:modulation} we obtain 
\[
\int V(x)|u(t_n,x)|^2\,dx\to 0,
\]
and hence $|x(t_n)|\to \infty$ by Lemma~\ref{xt-vs-potential} above.

We next suppose towards a contradiction that $|x(t_n)|\to\infty$ for some $t_n\to\infty$, but 
\begin{equation}\label{delta-bdd-below}
\delta(u(t_n))\geq c>0
\end{equation}
along some subsequence in $n$.  Note that by Lemma~\ref{xt-vs-potential} above, we have 
\begin{equation}\label{potential-to-zero}
\int_{\R^3}V(x)|u(t_n,x)|^2\,dx \to 0. 
\end{equation}

Using compactness (i.e. \eqref{u-is-compact}), we may now find $v_0\in H^1$ so that
\begin{equation}\label{CMT}
u(t_n,\cdot+x(t_n))\to v_0 \qtq{in}H^1
\end{equation}
along a further subsequence.  Using $H^1$ convergence, \eqref{delta-bdd-below}, and \eqref{potential-to-zero}, we deduce that
\[
M(v_0)=M(Q),\quad E_0(v_0)=E_0(Q),\qtq{and} \int |\nabla Q|^2\,dx  > \int |\nabla v_0|^2\,dx.
\]
Thus, by the main result of \cite{DR}, the solution $v$ to $(\text{NLS}_0)$ with initial data $v_0$ is global and either scatters as $t\to\infty$ or as $t\to-\infty$ (or both).  In either scenario, we will derive a contradiction to the fact that $u$ has infinite $L_{t,x}^5$-norm on $[0,\infty)\times\R^3$.  

\emph{Case 1.} First suppose $V$ satisfies \eqref{V1}--\eqref{V2}, and suppose that $v$ scatters as $t\to\infty$.  Arguing as in Proposition~\ref{P:embedding}, we can show that
\[
\tilde v_n(t,x) := v(t,x-x(t_n))
\]
defines an approximate solution to \eqref{nls} on $[0,\infty)$.  Indeed, we have that
\[
i\partial_t \tilde v_n+ \Delta \tilde v_n - V\tilde v_n + |\tilde v_n|^2 \tilde v_n = -V\tilde v_n,
\]
and we can show that
\[
\lim_{n\to\infty} \|V\tilde v_n\|_{L_t^4 L_x^{\frac65}([0,\infty)\times\R^3)} = 0
\]
just as we proved \eqref{444error}. As \eqref{CMT} yields
\[
\|u(t_n,x)-\tilde v_n(0)\|_{H^1} \to 0,
\]
an application of the stability result (Proposition~\ref{P:stab}) and the assumption that $v\in L_{t,x}^5([0,\infty)\times\R^3)$ therefore leads to 
\[
\|u(t_n+t)\|_{L_{t,x}^5([0,\infty)\times\R^3)}=\|u\|_{L_{t,x}^5([t_n,\infty)\times\R^3)}\lesssim 1
\]
for all large $n$, contradicting that the $L_{t,x}^5$-norm of $u$ is infinite on $[0,\infty)$.

If instead the solution $v$ scatters as $t\to-\infty$, the approximate solutions $\tilde v_n$ are instead approximate solutions on $(-\infty,0]$, and so the argument yields
\[
\|u(t_n+t)\|_{L_{t,x}^5((-\infty,0]\times\R^3)} = \|u\|_{L_{t,x}^5((-\infty,t_n]\times\R^3)} \lesssim 1
\] 
for large $n$, which is a contradiction.
 
\emph{Case 2.} Next, we suppose $V(x)=a|x|^{-2}$. In this case, the argument is basically the same as the one used to prove \cite[Theorem~6.1]{KMVZ}, which constructed scattering solutions to \eqref{nls} corresponding to initial data below the standard NLS scattering threshold and living far from the origin.  Accordingly, our presentation will be somewhat brief.

First suppose that $v$ scatters as $t\to\infty$.  We can then build approximate solutions to \eqref{nls} on $[0,\infty)$ as follows:  We let $P_n$ denote the Littlewood--Paley projection to frequencies less than $|x(t_n)|^\theta$ for some $0<\theta\ll 1$, and let $w_n$ denote the solution to ($\text{NLS}_0$) with initial data $P_n v_0$.  As $P_n\to Id$ strongly in $H^1$, we may apply the stability theory for ($\text{NLS}_0$) to deduce that $w_n$ are forward-global, scattering solutions for all large $n$. We then define $\chi_n$ to be a smooth function satisfying
\[
\chi_n(x) = \begin{cases} 0 & |x+x(t_n)|<\tfrac14|x(t_n)| \\ 1 & |x+x(t_n)| >\tfrac12 |x(t_n)|\end{cases} 
\]
and obeying the symbol bounds $|\partial^\alpha\chi_n(x)| \lesssim |x(t_n)|^{-|\alpha|}$.  Given $T>0$, we finally define the approximate solutions 
\begin{equation}\label{tilde-vnT}
\tilde v_{n,T}(t,x) = \begin{cases} [\chi_n w_n](t,x-x(t_n)) & 0 \leq t\leq T, \\ e^{-i(t-T)H}\tilde v_{n,T}(T) & t>T.\end{cases}  
\end{equation}
We then claim that the following hold:
\begin{align}
&\limsup_{T\to\infty}\limsup_{n\to\infty} \|\tilde v_{n,T}(0)-v_0(x-x(t_n))\|_{\dot H^{\frac12}} = 0, \label{embedding1}\\
&\limsup_{T\to\infty}\limsup_{n\to\infty}\ \bigl\{ \|\tilde v_{n,T}\|_{L_t^\infty H_x^1([0,\infty)\times\R^3)}+\|\tilde v_{n,T}\|_{L_{t,x}^5([0,\infty)\times\R^3)}\bigr\} \lesssim 1, \label{embedding2} \\
&\limsup_{T\to\infty}\limsup_{n\to\infty}\||\nabla|^{\frac12}[(i\partial_t-H)\tilde v_{n,T}+|\tilde v_{n,T}|^2\tilde v_{n,T}]\|_{N([0,\infty))} = 0,\label{embedding3}
\end{align}
where $N$ is as in \eqref{def:N}. 

We first note that \eqref{embedding1} follows by construction and the dominated convergence theorem; the essential fact here is that $|x(t_n)|\to\infty$, so that $\chi_n(x)\to 1$ for all $x\in\R^3$ and $P_n\to Id$ strongly in $H^1$.  Similarly, the bounds in \eqref{embedding2} follow from the $L_t^\infty H_x^1$ and $L_{t,x}^5$ bounds for $w_n$, along with Strichartz estimates for the linear propagator.

The proof of \eqref{embedding3}, although a bit more involved, follows as in the proof of \cite[(6.10)]{KMVZ}. Thus we will only sketch the main ideas.  First, on the region $t>T$, the error arises only from the nonlinearity, and hence it ultimately suffices to prove that 
\begin{equation}\label{approx-error1}
\lim_{T\to\infty}\limsup_{n\to\infty} \|e^{-it H_n}[\chi_n P_n v(T)]\|_{L_{t,x}^5((0,\infty)\times\R^3)} = 0,
\end{equation}
where the operator
\[
H_n := -\Delta + a|x+x(t_n)|^{-2}
\]
takes into account the failure of translation symmetry for $H$.  Using the facts that $|x(t_n)|\to\infty$ (so that $H_n$ converges to $-\Delta$ in the sense of \cite[Lemma~3.3]{KMVZZ}); that the solutions $w_n$ converge in $L_t^\infty H_x^1$ to $v$ on any fixed finite interval; and that $v$ scatters to some final state $v_+$, we can eventually derive \eqref{approx-error1} from the fact that
\[
\lim_{T\to\infty} \|e^{it\Delta}e^{iT\Delta} v_+\|_{L_{t,x}^5((0,\infty)\times\R^3)}=\lim_{T\to\infty}\|e^{it\Delta}v_+\|_{L_{t,x}^5((T,\infty)\times\R^3)}=0. 
\]

On the region $0\leq t\leq T$, we rely on the fact that the $w_n$ are solutions to ($\text{NLS}_0$).  The error therefore consists of three types of terms:
\begin{itemize}
\item[(i)] A nonlinear error due to the cutoff $\chi_n$, namely,
\[
[(\chi_n^3-\chi_n)|w_n|^2 w_n](t,x-x(t_n)).
\]
\item[(ii)] Linear errors due to the cutoff $\chi_n$, namely,
\[
[w_n\Delta \chi_n +2\nabla \chi_n\cdot\nabla w_n](t,x-x(t_n)).
\]
\item[(iii)] The error due to the potential, namely,
\[
-a|x|^{-2}[\chi_n w_n](t,x-x(t_n)).
\]
\end{itemize}

Term (iii) basically explains why the cutoff $\chi_n$ was introduced in the first place.  That is, the presence of the cutoff and the fact that $|x(t_n)|\to\infty$ guarantees that the potential $|x|^{-2}$ is pointwise small, which yields the required decay for this term. We also see the necessity of approximating with $w_n$ only on a finite interval $[0,T]$, as we must simply estimate the time integrals via H\"older. 

Term (ii), which requires a finite time interval for the same reason, enjoys decaying pointwise bounds arising from the derivative bounds for $\chi_n$ and the fact that $|x_n|\to\infty$.  However, this term additionally relies on the low-frequency cutoff $P_n$.  Indeed, the derivative of $w_n$ appears, while the stability theory requires us to estimate a further derivative.  Here we rely on persistence of regularity to obtain bounds for higher derivatives of $w_n$ that grow only like powers of $|x(t_n)|^\theta$.  In particular, choosing $\theta$ small enough, the pointwise bounds arising from the derivatives of $\chi_n$ are more than enough to obtain decay as $n\to\infty$.

Finally, term (i) can be controlled by relying on the space-time bounds for $w_n$ and the dominated convergence theorem, since $\chi_n^3 \to \chi_n\to 0$ on $\R^3$.  Thus we finally conclude that the error estimate \eqref{embedding3} holds.

With \eqref{embedding1}--\eqref{embedding3} in place, we may apply the stability result, Proposition~\ref{P:stab}, to deduce that for all large $n$, there exists a solution $v_n$ to \eqref{nls} that satisfies
\[
v_n(0)=v_0(\cdot-x(t_n))\qtq{and}\|v_n\|_{L_{t,x}^5([0,\infty)\times\R^3)}\lesssim 1. 
\]
Thus, recalling \eqref{CMT}, we have
\[
\|u(t_n,x)-v_n(0)\|_{H^1}\to 0 \qtq{as}n\to\infty,
\]
so that another application of Proposition~\ref{P:stab} implies
\[
\|u(t_n+t)\|_{L_{t,x}^5([0,\infty)\times\R^3)} = \|u\|_{L_{t,x}^5([t_n,\infty)\times\R^3)} \lesssim 1
\]
for all large $n$.  However, this contradicts the fact that the $L_{t,x}^5$-norm of $u$ is infinite on $[0,\infty)$. 

If instead the solution $v$ scatters as $t\to-\infty$, the approximate solutions $\tilde v_{n,T}$ (see \eqref{tilde-vnT} above) are defined in the analogous fashion on $(-\infty,0]$.  Repeating the arguments above, we construct solutions to \eqref{nls} with initial data $v_0(\cdot-x(t_n))$ that scatter backwards in time.  In this case, Proposition~\ref{P:stab} implies
\[
\|u(t_n+t)\|_{L_{t,x}^5((-\infty,0])} = \|u\|_{L_{t,x}^5((-\infty,t_n)\times\R^3)} \lesssim 1
\]
for all large $n$, again contradicting that the $L_{t,x}^5$-norm of $u$ is infinite. \end{proof}

Finally, we connect $\delta(t)$ to the standard virial quantity as follows.

\begin{lemma}\label{L:virial-versus-delta} There exists $c>0$ so that
\begin{equation}\label{virial-versus-delta}
F_\infty^0[u(t)]=\int |\nabla u(t,x)|^2 - \tfrac34|u(t,x)|^4\,dx \geq c\,\delta(t)\qtq{for all}t\in[0,\infty). 
\end{equation}
\end{lemma}

\begin{proof} We begin by connecting the left-hand side of \eqref{virial-versus-delta} to $\delta(t)$. By the Pohozaev identities for $Q$ (see, e.g. \cite[(3.2)]{KMVZ}), we have
\begin{equation}\label{poho}
\|Q\|_{\dot H^1}^2 = \tfrac34\|Q\|_{L^4}^4,\qtq{so that} 3E_0(Q)=\tfrac12\int |\nabla Q|^2\,dx. 
\end{equation}
We may also write
\[
\int |\nabla u|^2 - \tfrac34|u|^4\,dx = 3E_V(u) -  \int\tfrac12 |\nabla u|^2+\tfrac32 V(x)|u|^2\,dx,
\]
and hence (recalling $E_V(u)=E_0(Q)$)
\begin{equation}\label{virial-connect}
\int |\nabla u|^2 - \tfrac34|u|^4 \,dx = \tfrac12\delta(t) - \int V(x)|u|^2\,dx.
\end{equation}

Next, we show positivity for the left-hand side of \eqref{virial-connect}. In fact, using the sharp Gagliardo--Nirenberg inequality, \eqref{poho}, \eqref{delta-is-always-positive}, and \eqref{threshold}, we obtain
\begin{equation}\label{squeeze}
\|u\|_{L^4}^4\leq C_0\|u\|_{L^2}\|\nabla u\|_{L^2}^3 = \frac{\|Q\|_{L^4}^4}{\|\nabla Q\|_{L^2}^2} \cdot\frac{\|\nabla u\|_{L^2}}{\|\nabla Q\|_{L^2}}\cdot\|\nabla u\|_{L^2}^2< \tfrac{4}{3}\|\nabla u\|_{L^2}^2.
\end{equation}

Now suppose that \eqref{virial-versus-delta} fails. Then we may find a sequence $\{t_n\}\subset[0,\infty)$ so that 
\begin{equation}\label{vqtoz}
\int |\nabla u(t_n,x)|^2 - \tfrac34|u(t_n,x)|^4\,dx \leq \tfrac{1}{n}\delta(t_n). 
\end{equation}
In particular, the left-hand side of \eqref{vqtoz} tends to zero and hence (by \eqref{squeeze}) we obtain 
\[
\|\nabla u(t_n)\|_{L^2}\to \|\nabla Q\|_{L^2},
\]
which yields
\[
\delta(t_n)=-\int V(x)|u(t_n,x)|^2\,dx + o(1) \qtq{as}n\to\infty.
\]
Returning to \eqref{virial-connect} and recalling \eqref{vqtoz}, we find that $\delta(t_n)\to 0$.  However, by Proposition~\ref{P:modulation}, this implies that 
\[
\int V(x)|u(t_n,x)|^2\,dx \lesssim \delta(t_n)^2 \leq \tfrac{1}{4}\delta(t_n)
\]
for $n$ sufficiently large. Inserting this into \eqref{vqtoz} and \eqref{virial-connect} then implies
\[
\tfrac14\delta(t_n)\leq \tfrac{1}{n}\delta(t_n)
\]
for all $n$ sufficiently large, contradicting \eqref{delta-is-always-positive}.\end{proof}

\subsection{If $x(t)$ is bounded, then $x(t)$ is unbounded.}\label{BtoUB}

In this section, we show that $x(t)$ cannot be bounded.  Indeed, we will show the following:

\begin{proposition}\label{bounded-implies-unbounded} If $x(t)$ is bounded, then $x(t)$ is unbounded.
\end{proposition}

The key ingredient is the standard (truncated) virial identity. 

\begin{lemma}[Virial estimate]\label{L:standard-virial} For any $T>0$ and $\eps>0$, there exists $C_\eps>0$ so that
\begin{equation}\label{the-actual-estimate}
\int_0^T\delta(t)\,dt \lesssim \bigl[C_\eps+\sup_{t\in[0,T]}|x(t)|\bigr]\|u\|_{L_t^\infty H_x^1}^2 + \eps T.
\end{equation}
\end{lemma}

\begin{proof} Fix $T>0$. We will apply the virial identity on $[0,T]$ in the form \eqref{virial-id}.  In particular, we adopt the notation introduced in Section~\ref{S:virial}.  With $\eps>0$ given and $R>1$ to be determined below, we write
\begin{equation}\label{ftc1}
\tfrac{d}{dt}P_R[u] = F_\infty^0[u(t)] + F_R^V[u(t)]-F_\infty^0[u(t)], 
\end{equation}
where (recalling Lemma~\ref{L:virial-versus-delta}) 
\begin{align*}
|P_R[u(t)]| &=\biggl| 2\Im \int \bar u \nabla u \cdot\nabla w_R\,dx\biggr| \lesssim R\|u\|_{L_t^\infty H_x^1}^2, \\
F_\infty^0[u(t)] &\geq c\delta(t),
\end{align*}
and
\begin{align}
F_R^V[u(t)] - F_\infty^0[u(t)] & = -\int_{|x|>R} 8 |\nabla u|^2 - 6|u|^4 +4\Re \bar u_j u_k \partial_{jk}[w_R]\,dx \label{SVE1} \\
& \quad + \int_{|x|>R}  (-\Delta\Delta w_R)|u|^2 - |u|^4 \Delta w_R\,dx\label{SVE2}
 \\ & \quad + \int - 2|u|^2\nabla V\cdot\nabla w_R \,dx.\label{SVE3}
\end{align}

By compactness in $H^1$ (i.e. \eqref{u-is-compact}) and \eqref{V1} or \eqref{V3}, we may choose $C_\eps>0$ large enough that
\begin{equation}\label{virial-compact-smallness}
\int_{|x-x(t)|>C_\eps}|\nabla u|^2 + |x\cdot\nabla V|\,|u|^2+|x|^{-2}|u|^2+|u|^4\,dx \ll \eps.
\end{equation}
Thus, using the conditions on the weight $w_R$ specified in Section~\ref{S:virial} and choosing
\[
R=C_\eps+\sup_{t\in[0,T]}|x(t)|, \qtq{so that}\{|x|>R\}\subset\{|x-x(t)|>C_\eps\}\qtq{for}t\in[0,T],
\]
we may obtain
\begin{equation}\label{standard-virial-error-small}
|\eqref{SVE1}|+|\eqref{SVE2}| < \eps \qtq{uniformly for}t\in[0,T]. 
\end{equation}

For \eqref{SVE3}, we use the definition of $w_R$, the fact $V$ is repulsive, and \eqref{virial-compact-smallness} to obtain
\[
\eqref{SVE3} = \int -4x\cdot\nabla V |u|^2\,dx + \mathcal{O}\biggl(\int_{|x|>R}  |x\cdot\nabla V| |u|^2\,dx\biggr) \geq -\tfrac{1}{100}\eps.
\]

Thus, applying the fundamental theorem of calculus (i.e. integrating \eqref{ftc1} over $[0,T]$), we therefore obtain the estimate \eqref{the-actual-estimate}.   \end{proof}

\begin{proof}[Proof of Proposition~\ref{bounded-implies-unbounded}] Suppose that $x(t)$ is bounded.  By Lemma~\ref{L:standard-virial}, we deduce that for any $\eps>0$, \[
\tfrac{1}{T}\int_0^T\delta(t) dt \lesssim \tfrac{1}{T}C_\eps + \eps \qtq{for all}T>0. 
\]
Choosing a sequence of $\eps_n\to 0$ and a suitable sequence $T_n\to\infty$, we may now obtain a sequence $t_n\in[\tfrac12T_n,T_n]$ such that $\delta(t_n)\to 0$. Using Lemma~\ref{ximpliesdeltaimpliesx}, we deduce $|x(t_n)|\to\infty$.  \end{proof}

\subsection{If $x(t)$ is unbounded, then $x(t)$ is bounded.}\label{UBtoB}  In this section, we will show that $x(t)$ cannot be unbounded.  We prove:
\begin{proposition}\label{unbounded-implies-bounded} If $x(t)$ is unbounded, then $x(t)$ is bounded. 
\end{proposition}

This section will also require the use of a virial argument that yields control over $\delta(t)$.  However, in this setting, we will consider a `modulated' version of the virial identity that takes into account the possibility that $u$ is close to the orbit of $Q$.

\begin{lemma}[Modulated virial]\label{modulated-virial} Fix $\delta_1\in(0,\delta_0)$ sufficiently small.  Then for any $[t_1,t_2]\subset[0,\infty)$, we have
\[
\int_{t_1}^{t_2} \delta(t)\,dt \leq C_{\delta_1}\bigl[1+\sup_{t\in[t_1,t_2]}|x(t)|\bigr]\{\delta(t_1)+\delta(t_2)\}.
\]
\end{lemma}

\begin{proof}  Fix $R\geq 1$ to be chosen below.  We will apply the virial identity on $[t_1,t_2]$ in the form given in Corollary~\ref{C:virial}.  In particular, we adopt the notation introduced in Section~\ref{S:virial}.  We will apply this corollary with 
\[
\chi(t) = \begin{cases} 1 & \delta(t) < \delta_1 \\ 0 & \delta(t) \geq \delta_1.\end{cases}
\] 
On the support of $\chi$, Proposition~\ref{P:modulation} applies and we take $y(t)$ and $\theta(t)$ as in \eqref{modulation-identity}.  Recalling the notation of that section, as well as Lemma~\ref{L:virial-versus-delta}, we therefore have
\begin{equation}\label{ftc2}
\tfrac{d}{dt}P_R[u] = F_\infty^0[u(t)] + e(t) \geq c\delta(t)  + e(t),
\end{equation}
where
\begin{equation}\label{error-case1}
e(t) = F_R^V[u(t)]- F_\infty^0[u(t)] \qtq{if}\delta(t)\geq \delta_1
\end{equation}
and
\begin{equation}\label{error-case2}
\begin{aligned}
e(t) & = F_R^V[u(t)]-F_\infty^0[u(t)]   \\
& \quad - \bigl\{F_R^0[e^{i\theta(t)}Q(\cdot-y(t))] - F_\infty^0[e^{i\theta(t)}Q(\cdot-y(t))]\bigr\}
\end{aligned}\qtq{if}\delta(t)<\delta_1.
\end{equation}

To derive the desired estimate, we need to obtain bounds for $P_R[u(t_j)]$ as well as $e(t)$.

We first estimate $P_R[u(t_j)]$ for $j=1,2$.  In the present setting, we would like to exhibit a factor of $\delta(t_j)$.  If $\delta(t_j)\geq \delta_1$, this is straightforward, as
\[
|P_R(u(t))| =\biggl|2\Im\int \bar u\nabla u\cdot\nabla w_R\,dx \biggr| \lesssim R\|u\|_{L_t^\infty H_x^1}^2 \lesssim_u \tfrac{R}{\delta_1}\delta(t_j). 
\]
If instead $\delta(t_j)<\delta_1$, then we are in the setting of Proposition~\ref{P:modulation}.  In this case, we use the fact that $Q$ is real-valued and \eqref{modulation-bounds} to write
\begin{align*}
|P_R(u(t_j))| & = \biggl| 2\Im \int [\bar u\nabla u-e^{-i\theta(t_j)}Q(\cdot-y(t_j))\nabla [e^{i\theta(t_j)}Q(\cdot-y(t_j))]\cdot\nabla w_R\biggr| \,dx \\
& \lesssim R\{\|u\|_{L_t^\infty H_x^1}+\|Q\|_{H^1}\}\|u(t_j)-e^{i\theta(t_j)}Q(\cdot-y(t_j))\|_{H^1} \lesssim_u R\delta(t_j). 
\end{align*}

We turn to the error term, $e(t)$.  Here we would like to exhibit a small multiple of $\delta(t)$, so that the error term can be dominated by the main term.  

In the case of \eqref{error-case1} (when $\delta(t)\geq \delta_1$), this is once again straightforward.  In fact, the estimate is the same as in Lemma~\ref{L:standard-virial}. We let $\eps>0$ and (using compactness) choose $C_\eps$ sufficiently large so that
\[
\int_{|x-x(t)|>C_\eps}|\nabla u|^2 + |x\cdot\nabla V|\,|u|^2+ |x|^{-2}|u|^2 + |u|^4\,dx \ll \eps.
\]
We then choose
\begin{equation}\label{outside-ball}
R=C_\eps+9\sup_{t\in[t_1,t_2]}|x(t)|,\qtq{so}\{|x|>R\}\subset\{|x-x(t)|>C_\eps\}\qtq{for}t\in[t_1,t_2].
\end{equation}
Recalling the treatment of the error terms \eqref{SVE1}--\eqref{SVE3} above, we therefore obtain
\[
e(t)=F_R^V[u(t)] - F_\infty^0[u(t)]\geq -\eps \geq -\tfrac{\eps}{\delta_1}\delta(t)
\]
uniformly over $t\in[t_1,t_2]$ such that $\delta(t)\geq \delta_1$.

We turn to \eqref{error-case2}.  Here we essentially rely on the continuity of the functional 
\[
v\mapsto F_R^0[v]-F_\infty^0[v].
\]
In particular, temporarily adopting the notation
\[
Q(t) = e^{i\theta(t)}Q(\cdot-y(t)),
\]
we argue as for \eqref{SVE1}--\eqref{SVE3} above to write
\begin{align}
e(t) & =
-\int_{|x|>R} 8 [|\nabla u|^2-|\nabla Q(t)|^2]  - 6[|u|^4-|Q(t)|^4] \,dx \label{modulated-error1} \\
& \quad + 4\int_{|x|>R} \Re[ \bar u_j u_k- \partial_j \bar Q(t)\partial_kQ(t)] \partial_{jk}[w_R]\,dx \label{modulated-error2}  \\
& \quad + \int_{|x|>R} (-\Delta\Delta w_R)[|u|^2-|Q(t)|^2] - [|u|^4-|Q(t)|^4] \Delta w_R \,dx\label{modulated-error3}\\
& \quad +\int  -2|u|^2\nabla V\cdot\nabla w_R \,dx\label{modulated-error4}.
\end{align}
for $t\in[t_1,t_2]$ such that $\delta(t)<\delta_1$.  

We first handle \eqref{modulated-error1}--\eqref{modulated-error3}.  Observe that these may all be estimated by terms of the form
\[
\bigl\{\|u(t)\|_{H_x^1(|x|>R)}^k+\|Q(\cdot-y(t))\|_{H_x^1(|x|>R)}^k\bigr\}\|u(t)-e^{i\theta(t)}Q(\cdot-y(t))\|_{H_x^1},
\]
where $k\in\{1,3\}$.  Observing $\delta(t)<\delta_1$ implies $y(t)=x(t)$ and recalling \eqref{outside-ball} and \eqref{modulation-bounds}, this is further estimated by
\begin{align*}
\bigl\{ \|u(t)\|_{H_x^1(|x-x(t)|>C_\eps)}^k + \|Q\|_{H_x^1(|x|>C_\eps)}^k\bigr\} \delta(t) \lesssim \eps \delta(t),
\end{align*}
where we have chosen $C_\eps$ possibly even larger depending on $Q$.

It remains to deal with \eqref{modulated-error4}.  For this, we observe that by the definition of $w_R$ and the fact that $V$ is repulsive, we have
\begin{align}
\eqref{modulated-error4} & \geq \int_{|x|>R} -2|u|^2\nabla V\cdot\nabla w_R\,dx  \nonumber\\
& = \int_{|x|>R} -2[|u|^2 - |Q(t)|^2]\nabla V\cdot\nabla w_R\,dx \label{modulated-error5}\\
&  \quad -\int_{|x|>R}2|Q(t)|^2\nabla V\cdot\nabla w_R\,dx\label{modulated-error6}. 
\end{align}

Now, \eqref{modulated-error5} may be estimated as in \eqref{modulated-error1}--\eqref{modulated-error3}, leading to the bound
\[
|\eqref{modulated-error5}|\lesssim \eps\delta(t),
\] 
which is acceptable.  

It therefore remains to estimate \eqref{modulated-error6}. If $V$ satisfies \eqref{V1}--\eqref{V2}, we use $x\cdot\nabla V\in L^{3/2}$, while if $V$ satisfies \eqref{V3} we use the pointwise bound $|x\cdot\nabla V|\lesssim |x|^{-2}\lesssim 1$.  In this way, we arrive at the estimate
\[
|\eqref{modulated-error6}| \lesssim \|Q(\cdot-y(t))\|_{L^6(|x|>R)}^2 + \|Q(\cdot-y(t))\|_{L^2(|x|>R)}^2.
\]
Now, recall that $x(t)=y(t)$ for $\delta(t)<\delta_1$, and so by the definition of $R$ in \eqref{outside-ball}, we have
\[
\{|x|>R\}\subset\{|x-y(t)|>8|y(t)|\}
\]
for such $t$.  Then, recalling that $|Q(x)| \lesssim |x|^{-1} e^{-|x|}$ for $|x|\geq 1$, we obtain the following by explicit computation and the modulation bound \eqref{modulation-bounds}:
\begin{align*}
\|Q\|_{L^6(|x|>8|y(t)|)}^2& +\|Q\|_{L^2(|x|>8|y(t)|)}^2 \\&\lesssim e^{-8|y(t)|} \lesssim |y(t)|^{-4}e^{-4|y(t)|}\lesssim [\delta(t)]^2\lesssim\delta_1 \delta(t),
\end{align*}
which is acceptable. 

Applying the fundamental theorem of calculus (i.e. integrating \eqref{ftc2} over $[t_1,t_2]$) and collecting the estimates above, we obtain 
\begin{align*}
\int_{t_1}^{t_2}\delta(t)\,dt & \lesssim \tfrac{R}{\delta_1}\{\delta(t_1)+\delta(t_2)\} + \{\tfrac{\eps}{\delta_1}+\eps+\delta_1\}\int_{t_1}^{t_2} \delta(t)\,dt \\
& \lesssim \tfrac{1}{\delta_1}\bigl[C_\eps+\sup_{t\in[t_1,t_2]}|x(t)|\bigr]\{\delta(t_1)+\delta(t_2)\}+ \{\tfrac{\eps}{\delta_1}+\eps+\delta_1\}\int_{t_1}^{t_2} \delta(t)\,dt.
\end{align*}
Choosing $\eps=\eps(\delta_1)$ sufficiently small (and denoting $C_{\delta_1}=\tfrac{1}{\delta_1}C_\eps$), we obtain the desired estimate. \end{proof}

Roughly speaking, Lemma~\ref{ximpliesdeltaimpliesx} and Lemma~\ref{modulated-virial} show that if $x(t)$ becomes unbounded, then we can control integrals of the form $\int_I \delta(t)\,dt$ in terms of a \emph{small} multiple of $\sup_I |x(t)|$.  The last piece of the puzzle is therefore to show that $\int_I \delta(t)\,dt$ controls the variation of $x(t)$ on $I$, for then we can show that `$x(t)$ controls itself' and ultimately remains bounded.  Here we face a slight annoyance due to the piecewise definition of $x(t)$ (cf. \eqref{def:x(t)}).  Indeed, while the bound $|\dot y|\lesssim \delta$ provided by Proposition~\ref{P:modulation} is precisely the type input we want, it is not so clear on which parts of an arbitrary interval $I$ we should use $y(t)$ versus $x_0(t)$.  Nonetheless, with the aid of a few lemmas, we can establish the following:

\begin{proposition}\label{delta-controls-x} There exists $\delta_1>0$ sufficiently small and $C>0$ such that for any interval $[t_1,t_2]\subset[0,\infty)$, 
\begin{equation}\label{delta-controls-x-eq}
|x(t_2)-x(t_1)| \leq \tfrac{C}{\delta_1}\int_{t_1}^{t_2}\delta(t)\,dt + 2C. 
\end{equation}
\end{proposition}

We begin with a lemma that bounds the local variation of $x(t)$.

\begin{lemma}\label{L:local-constancy} There exists $C>0$ such that for all $t,s\in[0,\infty)$ with $|t-s|\leq 1$, we have
\[
|x(t)-x(s)| \leq C. 
\]
\end{lemma}

\begin{proof} By the Duhamel formula, Sobolev embedding, and uniform boundedness in $H^1$, we first observe that
\[
\|u(t)-u(s)\|_{L^2} \leq \int_{s}^t \| e^{-i(t-\tau)H}|u|^2 u(\tau)\|_{L^2}\,ds \leq |t-s|\|u\|_{L_t^\infty L_x^6}^3\lesssim|t-s|
\]
for any $s,t\in[0,\infty)$.  Next, given $\eps>0$ small, we can use compactness (i.e. \eqref{u-is-compact}) to find $C_\eps>0$ so that
\[
\sup_{t\in[0,\infty)} \|u(t)[1-\chi_\eps(t)]\|_{L^2}<\eps,\qtq{where}\chi_\eps(t)=\chi_{B_{C_\eps}(x(t))}. 
\]
In particular, by the triangle inequality, we have
\begin{equation}\label{lcfx-1}
\|u(t)\chi_\eps(t)-u(s)\chi_\eps(s)\|_{L^2} \leq C_u|t-s| + 2\eps \qtq{uniformly over}t,s\in[0,\infty).
\end{equation}

Now let $\eta>0$ to be determined below and suppose that $s_n,t_n\in[0,\infty)$ obey 
\[
|t_n-s_n|\leq \eta\qtq{but}|x(t_n)-x(s_n)|\to\infty.
\]
Then for $n$ sufficiently large, we have
\begin{equation}\label{lcfx-2}
\begin{aligned}
\|u(t_n)\chi_\eps(t_n)-u(s_n)\chi_\eps(s_n)\|_{L^2}^2 & = \|u(t_n)\chi_\eps(t_n)\|_{L^2}^2 + \|u(s_n)\chi_\eps(s_n)\|_{L^2}^2 \\
& \geq 2[M(u)-\eps]^2.
\end{aligned}
\end{equation}

Combining \eqref{lcfx-1} and \eqref{lcfx-2} yields a contradiction upon choosing $\eps$ and then $\eta$ sufficiently small (depending only on the $L_t^\infty H_x^1$-norm of $u$).  We conclude that there exists $C_0>0$ so that $|x(t)-x(s)|\leq C_0$ for all $|t-s|\leq \eta$, which implies the lemma with $C=\eta^{-1}C_0$. \end{proof}

Also useful is the fact that if $\delta(t)$ is very small at some $t\in[0,\infty)$, then there is an entire unit-sized interval containing $t$ where the modulation theory of Proposition~\ref{P:modulation} applies.  This is implied by the following lemma.

\begin{lemma}\label{L:local-constancy-2} There exists $\delta_1>0$ sufficiently small so that for all $t_0\in[0,\infty)$,
\[
\qtq{either} \inf_{t\in[t_0,t_0+1]} \delta(t) \geq \delta_1 \qtq{or} \sup_{t\in[t_0,t_0+1]}\delta(t)<\delta_0.
\]
\end{lemma}

\begin{proof} Suppose the statement is false. Then there exist $\tau_n\in[0,\infty)$, and $t_n,t_n'\in[\tau_n,\tau_n+1]$ such that
\[
\delta(t_n)\to 0 \qtq{and}\delta(t_n')\geq\delta_0. 
\]
Passing to a subsequence, we may assume $t_n-t_n'\to\tau\in[-1,1]$.  We also have that $x(t_n')$ is bounded and hence converges along a subsequence.  Indeed, this follows from Lemma~\ref{L:local-constancy} if $t_n'$ is bounded and from Lemma~\ref{ximpliesdeltaimpliesx} if $t_n'$ is unbounded.

Using compactness (i.e. \eqref{u-is-compact}) and convergence of $x(t_n')$, we therefore deduce that $u(t_n')$ converges strongly to some $v_0$ in $H^1$ along a further subsequence, which necessarily obeys $M(v_0)=M(Q)$, $E_V(v_0)=E_0(Q)$, and $\delta(v_0)\geq\delta_0$.  By Corollary~\ref{C:global} (and its proof), we deduce that the solution to \eqref{nls} with initial data $v_0$ is global and obeys $\delta(v(t))>0$ for all $t\in\R$.  By stability, we observe that $u(t_n'+\tau)\to v(\tau)$ strongly in $H^1$ and so $\delta(t_n'+\tau)\geq c_0$ for some $c_0>0$ and all large $n$.

On the other hand, we claim that $\| u(t_n) - u(t_n'+\tau)\|_{\dot H^1} \to 0$, so that (recalling $\delta(t_n)\to 0$) we obtain $\delta(t_n'+\tau)\to 0$, yielding a contradiction.  To see this, we first observe that by the local theory for \eqref{nls} (see e.g. \cite[Theorem~2.15]{KMVZ}), we may obtain that
\[
\|u\|_{L_t^5 H_a^{1,\frac{30}{11}}(I\times\R^3)} \lesssim \|u\|_{L_t^\infty H_x^1}
\] 
whenever the interval $I$ is sufficiently small depending on $\|u\|_{L_t^\infty H_x^1}$.  Consequently, writing $I_n$ for the interval between $t_n$ and $t_n'+\tau$, we may estimate as in the proof  of that result (i.e. the Duhamel formula, Strichartz, equivalence of Sobolev spaces, and Sobolev embedding) to obtain
\begin{align*}
\|u(t_n)-u(t_n'+\tau)\|_{\dot H^1} & \lesssim \| |u|^2 u\|_{L_t^2 \dot H_x^{1,\frac65}(I_n\times\R^3)} \\
& \lesssim |I_n|^{\frac{1}{10}} \|u\|_{L_t^5 L_x^6(I_n\times\R^3)}^2 \|u\|_{L_t^\infty \dot H_x^1}\to 0 \qtq{as}n\to\infty. 
\end{align*}
\end{proof}

We turn to the proof of Proposition~\ref{delta-controls-x}.

\begin{proof}[Proof of Proposition~\ref{delta-controls-x}] Let $[t_1,t_2]\subset[0,\infty)$ and let $\delta_1>0$ be as in Lemma~\ref{L:local-constancy-2}.  We first note that by Lemma~\ref{L:local-constancy}, we have
\[
|x(t_2)-x(t_1)| \leq |x(\lfloor t_2\rfloor)-x(\lceil t_1\rceil)| + 2C,
\]
where $\lfloor\cdot\rfloor$ and $\lceil\cdot\rceil$ denote the floor and ceiling, respectively. We now write
\[
|x(\lfloor t_2\rfloor)-x(\lceil t_1\rceil)| \leq \sum_{m} |x(m+1)-x(m)|, 
\]
where here and below the sums are over $m\in[\lceil t_1\rceil,\lfloor t_2\rfloor-1]\cap\mathbb{N}$.  We now consider a single interval $[m,m+1]$. Then, by Lemma~\ref{L:local-constancy-2}, we have that
\[
\qtq{either} \sup_{t\in[m,m+1]}\delta(t)<\delta_0\qtq{or}\inf_{t\in[m,m+1]}\delta(t)>\delta_1. 
\]
In the first case, we have $x(m+1)=y(m+1)$ and $x(m)=y(m)$ (cf. Proposition~\ref{P:modulation}), and hence by \eqref{modulation-bounds} we have
\[
|x(m+1)-x(m)| = |y(m+1)-y(m)| \leq \int_m^{m+1} |\dot y(t)|\,dt \leq C\int_m^{m+1}\delta(t)\,dt. 
\]
In the second case, we have from Lemma~\ref{L:local-constancy} that
\[
|x(m+1)-x(m)| \leq C\leq \tfrac{C}{\delta_1}\int_m^{m+1} \delta(t)\,dt. 
\]
Thus, recalling $\delta(t)>0$ for all $t$, 
\begin{align*}
|x(t_2)-x(t_1)|&\leq  2C+\tfrac{C}{\delta_1} \sum_m \int_m^{m+1}\delta(t)\,dt \\
&\leq 2C+ \tfrac{C}{\delta_1}\int_{\lceil t_1\rceil}^{\lfloor t_2\rfloor }\delta(t)\,dt \leq 2C + \tfrac{C}{\delta_1} \int_{t_1}^{t_2} \delta(t)\,dt.
\end{align*}
\end{proof}

Finally, we prove Proposition~\ref{unbounded-implies-bounded}.

\begin{proof}[Proof of \eqref{unbounded-implies-bounded}] Suppose that $x(t)$ is unbounded.  Then we may find a sequence $t_n\to\infty$ such that $|x(t_n)|\to\infty$ and 
\[
|x(t_n)|=\sup_{t\in[0,t_n]}|x(t)|.
\] 

By Lemma~\ref{ximpliesdeltaimpliesx}, we have that $\delta(t_n)\to 0$, and hence there exists $N\in\mathbb{N}$ so that 
\[
\delta(t_n)<\tfrac{CC_{\delta_1}}{10\delta_1} \qtq{for all}n\geq N,
\]
where $\delta_1$ is as in Lemma~\ref{L:local-constancy-2}, $C_{\delta_1}$ is as in Lemma~\ref{modulated-virial}, and $C$ is as in Proposition~\ref{delta-controls-x}.  Then, combining Proposition~\ref{delta-controls-x} and Lemma~\ref{modulated-virial}, we derive
\begin{align*}
|x(t_n)-x(t_N)| & \leq \tfrac{C}{\delta_1}\int_{t_N}^{t_n}\delta(t)\,dt + 2C \\
& \leq \tfrac{CC_{\delta_1}}{\delta_1}[1+\sup_{t\in[t_N,t_n]}|x(t)|]\{\delta(t_N)+\delta(t_n)\} + 2C\\
& \leq \tilde C + \tfrac12 |x(t_n)|
\end{align*}
for some $\tilde C>0$. Therefore
\[
|x(t_n)| \leq |x(t_N)| + 2\tilde C \qtq{for all}n\geq N,
\]
which implies that $x(t_n)$ is bounded, a contradiction. \end{proof}

\subsection{Conclusion}\label{S:conclusion} In the preceding sections, we have shown that if Theorem~\ref{T} fails, then we may obtain a `blowup' solution to \eqref{nls} parametrized by a spatial center $x(t)$.  This spatial center proves its own undoing: if it is bounded, then it must be unbounded, and vice versa.  We therefore reach a contradiction and conclude that Theorem~\ref{T} holds.

We would also like to point out that the argument above may be simplified at several points if one restricts attention only to the case of the inverse-square potential, particularly in the setting of the virial arguments. This is due to the fact that in this case, the quantity 
\[
V(x)|u|^2\qtq{coincides with}-x\cdot\nabla V(x)|u|^2.
\]
The first quantity appears in the definition of $\delta$ (and the $\dot H_V^1$-norm/energy), while the second arises in the virial identities.  Some simplifications in this case are not surprising, as the virial identity is based off of the scaling symmetry, which is preserved in the special case of the inverse-square potential. In particular, a result such as Lemma~\ref{L:virial-versus-delta} becomes unnecessary in this special case, as the quantity arising in the virial identity \emph{equals} a constant times $\delta(t)$.  Another key simplification occurs in the treatment of the error term \eqref{modulated-error4} in the modulated virial argument.  Indeed, in the inverse-square case, this particular term reduces to $\int V(x)|u|^2\,dx$, which is already controlled by the acceptable term $[\delta(t)]^2$.  In particular, the modulation bound $|y(t)|^{-2}e^{-2|y(t)|}\lesssim \delta(t)$ obtained in Lemma~\ref{L:boundsII} and the reliance on explicit exponential bounds for $Q$ are no longer necessary.

\section{Modulation analysis}\label{S:modulation} 

In this section we prove the modulation result, Proposition~\ref{P:modulation}, which we reproduce as Proposition~\ref{P:modulation2} below.  We recall the functional
\[
\delta(v):= \int_{\R^3} |\nabla Q|^2\,dx - \int_{\R^3} |\nabla v|^2 + V(x)|v|^2\,dx, 
\]
where $Q$ is the standard NLS ground state.  We suppose throughout this section that $u(t)$ is a solution to \eqref{nls} obeying
\begin{equation}\label{MMEE}
M(u) = M(Q), \quad E_V(u) = E_0(Q),\qtq{and} \delta(u_0)>0.
\end{equation}
In particular, by Corollary~\ref{C:global} (and its proof), we obtain that $u(t)$ is global and uniformly bounded in $H^1$, with $\delta(t):=\delta(u(t))>0$ for all $t\in\R$. Given a small parameter $\delta_0>0$, we define the set
\[
I_0 = \{t\in[0,\infty):\delta(t)<\delta_0\},
\]
which (by continuity of the flow in $H^1$) is relatively open in $[0,\infty)$.  We then have the following description of $u(t)$ on the set $I_0$.

\begin{proposition}[Modulation]\label{P:modulation2} The following holds for any $\delta_0>0$ sufficiently small: There exist $\theta:I_0\to\R$ and $y:I_0\to\R^3$ so that $u(t)$ admits the decomposition
\begin{equation}\label{modulation-identity1}
u(t) = e^{i\theta(t)}[Q(x-y(t))+g(t)] \qtq{for}t\in I_0,
\end{equation}
with
\begin{equation}\label{modulation-bounds2}
\frac{e^{-2|y(t)|}}{|y(t)|^2}+ |\dot y(t)| + \biggl[\int V(x)|u(t,x)|^2\,dx\biggr]^{\frac12} \lesssim \delta(t) \sim \|g(t)\|_{H^1}
\end{equation}
for all $t\in I_0.$
\end{proposition}

The starting point is to show (in a non-quantitative way) that if $v\in H^1$ obeys $M(v)=M(Q)$ and $E_V(v)=E_0(Q)$, and $\delta(v)$ is small, then $v$ must be close to the orbit of $Q$ .  Here the orbit of $Q$ refers to the set
\[
\{e^{i\theta}Q(\cdot-y):\theta\in\R,\ y\in\R^3\}.
\]

\begin{lemma}[Modulation, non-quantitative version]\label{L:mod-nq} Suppose $\{v_n\}\subset H^1$ obeys 
\[
M(v_n)=M(Q),\quad E_V(v_n)=E_0(Q), \qtq{and} \delta(v_n)\to 0.
\]
Then there exist $\theta_n\in\R$ and $|y_n|\to\infty$ so that
\[
\|v_n- e^{i\theta_n}Q(\cdot-y_n)\|_{H^1} \to 0.
\]
\end{lemma}

\begin{proof} Arguing as in the proof of Corollary~\ref{C:global}, we see that if
\[
E_V(v_n)=E_0(Q) \qtq{and} \|v_n\|_{\dot H_V^1} \to \|Q\|_{\dot H^1}, 
\]
then $\|v_n\|_{L^4} \to \|Q\|_{L^4}$.  Thus, by the sharp Gagliardo--Nirenberg inequality \eqref{sharp-GN} (with $V=0$), 
\[
C_0^{-1}\leq \frac{\|v_n\|_{L^2}\|v_n\|_{\dot H^1}^3}{\|v_n\|_{L^4}^4} \leq \frac{\|v_n\|_{L^2} \|v_n\|_{\dot H_V^1}^3}{\|v_n\|_{L^4}^4} \to \frac{\|Q\|_{L^2}\|Q\|_{\dot H^1}^3}{\|Q\|_{L^4}^4} = C_0^{-1},
\]
In particular, $v_n$ is an optimizing sequence for sharp Gagliardo--Nirenberg, which (together with the mass constraint) yields
\[
e^{-i\theta_n}v_n(\cdot+y_n)\to Q \qtq{for some}\theta_n\in\R,\quad y_n\in\R^3.
\]
Note that the above inequalities also show that 
\[
\|v_n\|_{\dot H_V^1}-\|v_n\|_{\dot H^1} \to 0,\qtq{i.e.} \int V(x)|v_n(x)|^2\,dx \to 0. 
\]
This in turn implies
\[
\int V(x) |Q(x-y_n)|^2\,dx \to 0,\qtq{and hence}|y_n|\to\infty. 
\]
\end{proof}
%
%

This result is still far from what is claimed in Proposition~\ref{P:modulation2}.  Indeed, Proposition~\ref{P:modulation2} claims that if $\delta(t)$ is small, then by choosing $(\theta(t),y(t))$ appropriately we may make the $H^1$-norm of the error term
\begin{equation}\label{def:g}
g(t) = e^{-i\theta(t)}[u(t)-e^{i\theta(t)}Q(\cdot-y(t))]
\end{equation}
comparable to $\delta(t)$ itself.  The problem of obtaining bounds on $g(t)$ is reminiscent of the problem of orbital stability (see e.g. \cite{Weinstein}).  Accordingly, as a first attempt we might try to combine the mass and energy as a kind of Lyapunov functional.  In particular, a direct computation (with more details provided in Lemma~\ref{L:boundsI} below) yields
\begin{equation}\label{Lyapunov}
\begin{aligned}
0 &= E_V(u)+M(u) - [E_0(Q)+M(Q)] \nonumber\\
&  =  B(g(\cdot+y),g(\cdot+y)) + \tfrac12\int V(x)|u|^2\,dx + \mathcal{O}\{\|g\|_{H^1}^3 + \|g\|_{H^1}^4\}, 
\end{aligned}
\end{equation}
where $B(\cdot,\cdot)$ is the usual bilinear form arising from the linearization around the ground state, that is,
\begin{equation}\label{Bgg}
\begin{aligned}
B(g,g) & = \int \tfrac12 |\nabla g|^2 + \tfrac12 |g|^2 - (\tfrac32g_1^2+\tfrac12 g_2^2)Q^2\,dx \\
& =: \tfrac12\langle L_+ g_1,g_1\rangle + \tfrac12\langle L_- g_2,g_2\rangle,
\end{aligned}
\end{equation}
where $g=g_1+ig_2$. 

Continuing the analogy with proofs of orbital stability, we might now hope to obtain bounds on $g$ by imposing conditions on $(\theta,y)$ that yield coercivity in the quadratic term.  Evidently, this is impossible.  Indeed, noting that we may already arrange that $\|g\|_{H^1}\ll 1$, we see from \eqref{Lyapunov} that
\[
\|g\|_{H^1}^2 + \int V(x)|u|^2\,dx \leq 0.
\]
We may also observe that the subspaces on which $B(\cdot,\cdot)$ are known to be positive involve at least five constraints (see e.g. \cite{DR}), whereas $(\theta,y)\in \R^4$ only afford us four.  Nonetheless, these parameters will be sufficient to impose orthogonality to the kernels of $L_{\pm}$, which are well-known:
\begin{equation}\label{kernels}
\ker L_- = \text{span}\{Q\}\qtq{and} \ker L_+ = \text{span}\{\partial_j Q\}_{j=1}^3
\end{equation}
(see e.g. \cite{Weinstein}).  We can then further decompose $g$ into a part belonging to a positive subspace of $B(\cdot,\cdot)$ and a remainder term.  Essentially, the size of the remainder dictates what estimates we may obtain for $g$. 

Our first task is therefore to show that we may choose modulation parameters to impose orthogonality to the subspaces appearing in \eqref{kernels}.

\begin{lemma}[Modulation, with orthogonality]\label{L:modulation-orthogonality} If $\delta_0>0$ is sufficiently small, then we may define functions $\theta:I_0\to\R$ and $y:I_0\to\R^3$ so that
\begin{equation}\label{nq-small}
\|u(t) - e^{i\theta(t)}Q(\cdot-y(t))\|_{H^1}\ll 1 
\end{equation}
and 
\begin{equation}\label{orthogonality-conditions}
\Im\langle e^{i\theta(t)}Q(\cdot-y(t),u(t)\rangle = \Re\langle e^{i\theta(t)}\partial_j Q(\cdot-y(t)),u(t)\rangle =  0
\end{equation}
for $j\in\{1,2,3\}$. \end{lemma}

\begin{proof}  Let $\eps>0$ be a small parameter to be specified below.  By Lemma~\ref{L:mod-nq}, for $\delta_0=\delta_0(\eps)$ sufficiently small and $t\in I_0$, we may find $(\theta_0(t),y_0(t))\in\R^4$ such that 
\begin{equation}\label{initial-small}
\|u(t) - e^{i\theta_0(t)}Q(\cdot-y_0(t))\|_{H^1}<\eps.
\end{equation}
We will construct the parameters $(\theta(t),y(t))$ via the implicit function theorem.  To this end, we define the function 
\[
\Phi:H^1\times\R^4\to\R^4
\] 
by
\[
\Phi(v,z) = \bigl(\Im\langle e^{i\theta}Q(\cdot-y),v\rangle,\ \Re \langle e^{i\theta}\nabla Q(\cdot-y),v\rangle\bigr),\qtq{where} z=(\theta,y).
\]
Setting 
\[
(v_0,z_0)=(v_0(t),z_0(t)):=(e^{i\theta_0(t)}Q(\cdot-y_0(t)),\theta_0(t),y_0(t)),
\]
we observe that by construction, we have
\begin{equation}\label{curve-of-zeros}
\Phi(v_0,z_0) \equiv 0. 
\end{equation}

Our task is therefore to compute the derivatives $\partial_{z_k}\Phi_j$ and evaluate at $(v_0,z_0)$.

First,
\[
\partial_{z_1}\Phi_1 = -\Re\langle e^{i\theta}Q(\cdot-y),v\rangle \implies \partial_{z_1}\Phi_1|_{(v_0,z_0)}=-\|Q\|_{L^2}^2. 
\]
On the other hand, for $j\in\{2,3,4\}$, we compute
\[
\partial_{z_1} \Phi_j = \Im\langle e^{i\theta}\partial_{j-1} Q(\cdot-y),v\rangle \implies \partial_{z_1}\Phi_j|_{(v_0,z_0)}=0. 
\] 

Next, for $k\in\{2,3,4\}$, 
\[
\partial_{z_k}\Phi_1 = -\Im\langle e^{i\theta}\partial_{k-1}Q(\cdot-y),v\rangle \implies \partial_{z_k}\Phi_1|_{(v_0,z_0)}=0.
\]
On the other hand $j,k\in\{2,3,4\}$, we use the fact that $Q$ is radial (so that $\langle \partial_j Q,\partial_k Q\rangle =0$ for $j\neq k$) to obtain
\[
\partial_{z_k} \Phi_j =-\Re\langle e^{i\theta}\partial_{k-1}\partial_{j-1}Q(\cdot-y),v\rangle \implies \partial_{z_k}\Phi_j|_{(v_0,z_0)} = -\delta_{jk}\|\nabla Q\|_{L^2}^2.
\] 

It follows that $(\partial_{z_k}\Phi_j)|_{(v_0,z_0)}$ is independent of $t\in I_0$ and boundedly invertible.  We may therefore apply the implicit function theorem to the entire family of zeros of $\Phi$ given in \eqref{curve-of-zeros}.  In particular, choosing $\eta=\eta(Q)>0$ and $\eps=\eps(\eta)>0$ sufficiently small, we obtain that for each $t\in I_0$ there exists a function
\[
\zeta_t: B_\eps(e^{i\theta_0(t)}Q(\cdot-y_0(t)))\subset H^1\to B_\eta((\theta_0(t),y_0(t))\subset\R^4
\]
so that 
\[
\Phi(v,\zeta_t(v)) =0 \qtq{for all}v\in B_\eps(e^{i\theta_0(t)}Q(\cdot-y_0(t)). 
\]

As \eqref{initial-small} yields $u(t)\in B_\eps(e^{i\theta_0(t)}Q(\cdot-y_0(t))$ for all $t\in I_0$, we may choose
\[
(\theta(t),y(t))=\zeta_t(u(t)). 
\] 
With this choice we obtain the desired orthogonality conditions \eqref{orthogonality-conditions}, while \eqref{initial-small} and the fact that $|(\theta_0(t),y_0(t))-(\theta(t),y(t))|\ll1$ yield \eqref{nq-small}. \end{proof}

Defining $(\theta(t),y(t))$ as in Lemma~\ref{L:modulation-orthogonality} and
\begin{equation}\label{g-again}
g(t) = g_1(t)+ig_2(t) = e^{-i\theta(t)}[u(t)-e^{i\theta(t)}Q(\cdot-y(t))],
\end{equation}
we therefore have the following orthogonality conditions for $g$:
\begin{equation}\label{orthogonality-g}
\langle g_2(t),Q(\cdot-y(t))\rangle = \langle g_1(t),\partial_j Q(\cdot-y(t))\rangle \equiv 0 \qtq{for}j\in\{1,2,3\}.
\end{equation}

We turn to the problem of establishing quantitative estimates.

\begin{lemma}[Bounds, part I]\label{L:boundsI} Let $(\theta(t),y(t))$ be as in Lemma~\ref{L:modulation-orthogonality} and $g(t)$ be as in \eqref{g-again}.  Then
\begin{equation}\label{bounds-part1}
\biggl[\int V(x)|u(t,x)|^2\,dx\biggr]^{\frac12} \lesssim \delta(t) \sim \|g(t)\|_{H^1}
\end{equation}
for $t\in I_0$. 
\end{lemma}

\begin{proof}  As described above, the general strategy is to combine the mass and energy into a Lyapunov functional, which leads to the consideration of the bilinear form $B(\cdot,\cdot)$.  To this end, let us first provide the details leading to the identity \eqref{Lyapunov}.  Using the gauge invariance and translation invariance of $E_0(\cdot)$, we first write
\begin{align*}
E_0(u)-E_0(Q)& = \int\tfrac 12|\nabla g|^2 + \nabla g_1\cdot\nabla Q(\cdot-y)-g_1 [Q(\cdot-y)]^3\,dx \\
& \quad -\int \tfrac32 g_1^2 [Q(\cdot-y)]^2 +\tfrac12 g_2^2[Q(\cdot-y)]^2\,dx \\
& \quad - \int \tfrac14 |g|^4 + Q(\cdot-y)[g_1^3+g_1 g_2^2]\,dx.
\end{align*}
Similarly,
\[
M(u)-M(Q) = \int \tfrac12 |g|^2 + g_1 Q(\cdot-y)\,dx. 
\]
As integration by parts and \eqref{elliptic} (the equation for $Q$) yield
\[
\int g_1 Q(\cdot-y) + \nabla g_1 \cdot\nabla Q(\cdot-y) - g_1[Q(\cdot-y)]^3\,dx = 0,  
\]
we may use \eqref{MMEE} (equality of mass/energy) and \eqref{nq-small} to obtain the identity \eqref{Lyapunov}, which we reproduce here: 
\begin{equation}\label{Lyapunov2}
0=B(g(\cdot+y),g(\cdot+y)) + \tfrac12\int V(x)|u|^2\,dx + \mathcal{O}(\|g\|_{H^1}^3).
\end{equation}

In light of the discussion preceding Lemma~\ref{L:modulation-orthogonality}, we cannot expect coercivity for the quadratic term in $g$.  We can, however, further decompose $g$ into a piece belonging to the positive subspace of $B$ and a remainder term.  For this, we recall \cite[Proposition~2.7]{DR}, which yields coercivity for all $h=h_1+ih_2\in H^1$ satisfying $h_1\perp \ker\{L_+\}$, $h_2\perp \ker\{L_-\}$, and the additional orthogonality condition $h_1\perp\Delta Q$ (see \eqref{Bgg} above for the definition of $L_\pm$).  

In particular, if we define
\begin{equation}\label{def:h}
g=\alpha Q(\cdot-y) + h, \qtq{where} \alpha=\frac{\langle g_1(\cdot+y),\Delta Q\rangle}{\langle Q,\Delta Q\rangle}\in\R,
\end{equation}
we may observe that $h_1(\cdot+y)$ is still orthogonal to $\ker\{L_+\}$ (since $\langle Q,\nabla Q\rangle =0$) but additionally satisfies $h_1(\cdot+y)\perp \Delta Q$.  That is, $h(\cdot+y)$ belongs to the positive subspace of $B$.  

Our immediate goal will be to obtain estimates for the term $h$ and the coefficient $\alpha$.  We begin with the simple observation that
\begin{equation}\label{alpha-is-small}
|\alpha|\lesssim \|g\|_{L^2}\ll 1. 
\end{equation}

Next, we observe that by construction, we have
\[
B(h(\cdot+y),h(\cdot+y))\gtrsim \|h(\cdot+y)\|_{H^1}^2 = \|h\|_{H^1}^2,
\]
and hence expanding $B(g(\cdot+y),g(\cdot+y))$ in \eqref{Lyapunov2} yields
\begin{equation}\label{Lyapunov3}
\|h\|_{H^1}^2 + \int V(x)|u|^2\,dx  \lesssim \alpha^2+  |\alpha\langle L_+ Q,h_1(\cdot+y)\rangle|+ \mathcal{O}(\|g\|_{H^1}^3).
\end{equation}
Using $h_1\perp \Delta Q$, we find that 
\[
-2\langle Q^3,h_1(\cdot+y)\rangle = \langle L_+ Q,h_1(\cdot+y)\rangle = \langle Q,h_1(\cdot+y)\rangle-3\langle Q^3,h_1(\cdot+y)\rangle, 
\]
so that
\[
\langle L_+ Q,h(\cdot+y)\rangle = \langle Q,h_1(\cdot+y)\rangle= \langle Q(\cdot-y),h_1\rangle. 
\]
Thus, recalling \eqref{def:h} and \eqref{alpha-is-small}, we may continue from \eqref{Lyapunov3} to obtain
\begin{equation}\label{Lyapunov4}
\|h\|_{H^1}^2 + \int V(x)|u|^2\,dx \lesssim \alpha^2 +  |\alpha\langle Q(\cdot-y),h_1\rangle| + \mathcal{O}\bigl(\|h\|_{H^1}^3\bigr).
\end{equation}

At this point, we need to estimate the inner product appearing in \eqref{Lyapunov4}.  We can do so by exploiting the mass constraint.  In particular, recalling
\begin{equation}\label{masses-equal-expand}
M(Q)=M(u)=M(Q(\cdot-y)+g)=M([1+\alpha]Q(\cdot-y)+h),
\end{equation}
we derive that
\begin{equation}\label{mass-constraint-identity}
\alpha^2\|Q\|_{L^2}^2 + 2\alpha\|Q\|_{L^2}^2 + 2\langle Q(\cdot-y),h_1\rangle + \|h\|_{L^2}^2=0,
\end{equation}
which (recalling \eqref{alpha-is-small}) yields the estimate
\[
|\langle Q(\cdot-y,h\rangle| \lesssim |\alpha| + \|h\|_{L^2}^2. 
\]
Inserting this into \eqref{Lyapunov4} and applying Young's inquality yields
\[
\|h\|_{H^1}^2 + \int V(x)|u|^2\,dx \lesssim \alpha^2 + \mathcal{O}(\|h\|_{H^1}^3),
\]
so that
\begin{equation}\label{s.p.1}
\|h\|_{H^1}^2 + \int V(x)|u|^2\,dx \lesssim \alpha^2. 
\end{equation}

Returning to \eqref{mass-constraint-identity} and recalling \eqref{alpha-is-small}, we can also now observe that
\begin{equation}\label{s.p.2}
|\alpha| \lesssim \|h\|_{H^1}. 
\end{equation}
Combining \eqref{def:h}, \eqref{alpha-is-small}, \eqref{s.p.1}, and \eqref{s.p.2}, we therefore obtain
\[
\|g\|_{H^1} \sim |\alpha|\sim \|h\|_{H^1} \qtq{and}\int V(x)|u|^2\,dx \lesssim \alpha^2. 
\]

In light of the above estimates, it remains only to show that $|\alpha|\sim \delta$. For this, we recall the orthogonality condition $\langle \nabla Q(\cdot-y),\nabla h_1\rangle =0$ and write
\begin{align*}
\delta(t) & = \int |\nabla Q|^2\,dx - \int V(x)|u|^2 - \int (1+\alpha)^2|\nabla Q(\cdot-y)|^2 + |\nabla h|^2 \,dx \\
& = -2\alpha \int |\nabla Q|^2 - \alpha^2 \int |\nabla Q|^2 - \int V(x)|u|^2\,dx - \int |\nabla h|^2\,dx \\
& = -2\alpha \int |\nabla Q|^2\,dx + \mathcal{O}(\alpha^2),
\end{align*}
which yields the result. \end{proof}

We turn to the estimate on $y(t)$ itself. 

\begin{lemma}[Bounds, part II]\label{L:boundsII} If $\delta_0=\delta_0(V)$ is sufficiently small, then
\[
\frac{e^{-2|y(t)|}}{|y(t)|^2}\lesssim\delta(t) \qtq{for all}t\in I_0. 
\]
\end{lemma}

\begin{proof} We rely on many of the estimates given in the previous lemma; in particular, we recall the quantities $\alpha$ and $h$ given in \eqref{def:h}.

First observe that given a non-zero potential $V$ satisfying \eqref{V1}--\eqref{V2} or \eqref{V3}, we may find $b>0$, $c>0$, and $R\geq 1$ such that
\[
|S|\geq c,\qtq{where}S=S_{b,R}=\{x:|x|\leq R\qtq{and}|V(x)|\geq b\}.
\]
Thus, using \eqref{bounds-part1}, we may write
\[
\int_S |u(t,x)|^2\,dx \leq \tfrac{1}{b}\int V(x)|u(t,x)|^2\,dx \lesssim \tfrac{1}{b}[\delta(t)]^2 \lesssim \delta(t)
\]
for $\delta_0=\delta_0(V)$ small enough.  Next, we expand $u$ (as in \eqref{masses-equal-expand}) to write 
\[
|u|^2 = |Q(\cdot-y)|^2+(\alpha^2+2\alpha)|Q(\cdot-y)|^2+|h|^2+2Q(\cdot-y)h_1. 
\]
By the estimates obtained in the previous lemma, this leads to
\begin{equation}\label{555at}
\int_S |Q(x-y(t))|^2\,dx \lesssim\delta(t). 
\end{equation}

Now observe from Lemma~\ref{L:mod-nq} that for $\delta_0$ chosen possibly even smaller, we may guarantee that 
\[
\delta(t)<\delta_0\implies |y(t)|\geq 2R.
\]
This guarantees that $|x-y(t)|\geq R\geq 1$ for all $y\in I_0$. Observing also that $|x-y(t)|\leq |y(t)|+R$, we may continue from \eqref{555at} and use
\[
|Q(x)|\gtrsim |x|^{-1}e^{-|x|}\qtq{for} |x|\geq 1
\]
(see \cite{Cazenave}) to obtain
\[
ce^{2R} |y(t)|^{-2} e^{-2|y(t)|} \lesssim\int_S |Q(x-y(t))|^2\,dx \lesssim \delta(t).
\]
\end{proof}

To complete the proof of Proposition~\ref{P:modulation2}, it remains only to establish the desired estimate for $|\dot y|$. The basic strategy is to differentiate the orthogonality conditions. 

\begin{lemma}[Bounds, part III]\label{L:boundsIII} Let $(\theta(t),y(t))$ be as in Lemma~\ref{L:modulation-orthogonality} and $g(t)$ be as in \eqref{g-again}. Then
\[
|\dot y(t)| \lesssim \delta(t)\qtq{for}t\in I_0.
\]
\end{lemma}

\begin{proof} Using
\[
g(t) = u(t) - e^{i\theta(t)}Q(\cdot-y(t))
\]
and the equations
\[
i\partial_t u + \Delta u - V(x)u + f(u) = 0,\quad -Q+\Delta Q + f(Q) = 0,
\]
with $f(z)=|z|^2 z$, we derive the following evolution equation:
\begin{equation}\label{nls-g}
\begin{aligned}
i&\partial_tg + \Delta g - \dot\theta g \\ 
& \quad+ Q(x-y)-\dot\theta Q(x-y) - \dot y\cdot\nabla Q(x-y) \\
&\quad  - Ve^{-i\theta}u + f(e^{-i\theta}u)-f(Q(\cdot-y)) =0.\end{aligned}
\end{equation}

We will first obtain the estimate
\begin{equation}\label{theta-bd}
|\dot\theta|\lesssim 1 + |\dot y|\,\|g\|_{H^1}. 
\end{equation}
To isolate $\dot\theta$, we multiply \eqref{nls-g} by $Q(\cdot-y)$, integrate, and take the real part.  Using the orthogonality conditions \eqref{orthogonality-g}, the fact that $\langle Q,\nabla Q\rangle =0$, and \eqref{bounds-part1},  we obtain 
\begin{align}
|\dot\theta| \lesssim_Q |\Re\langle i\partial_t g,Q(\cdot-y)\rangle| + |\dot\theta|\,\|g\|_{H^1} + 1 +\mathcal{O}(\|g(t)\|_{H^1}).
\end{align}
As $\|g\|_{H^1}\ll1$, it suffices to estimate the term involving $\partial_t g$.  To this end, we observe that from the orthogonality conditions \eqref{orthogonality-g}, we have
\begin{align*}
\Re\langle i\partial_t g,Q(\cdot-y)\rangle = \Im \langle \partial_t g,Q(\cdot-y)\rangle =  -\dot y\Im\langle g,\nabla Q(\cdot-y)\rangle,
\end{align*}
so that this term is $\mathcal{O}(|\dot y|\,\|g\|_{H^1})$. 

We turn to the estimate of $\dot y$.  To isolate a component $\dot y_j$, we multiply \eqref{nls-g} by $\partial_j Q(\cdot-y)$ and take the imaginary part. Recalling $\langle \partial_j Q,\partial_k Q\rangle=0$ for $j\neq k$, integrating by parts in the $\Delta g$ term, and using \eqref{bounds-part1} and \eqref{theta-bd}, we obtain 
\begin{align*}
|\dot y_j| & \lesssim |\Im\langle i\partial_t g,\partial_j Q(\cdot-y)\rangle| + (1+|\dot\theta|)\|g\|_{H^1}\\
& \lesssim  |\Im\langle i\partial_t g,\partial_j Q(\cdot-y)\rangle| + \|g\|_{H^1} + |\dot y|\,\|g\|_{H^1}^2. 
\end{align*}
It therefore remains to estimate the term involving $\partial_t g$, for which we again rely on the orthogonality conditions \eqref{orthogonality-g}.  We write
\[
\Im\langle i\partial_t g,\partial_j Q(\cdot-y)\rangle = \Re\langle \partial_t g,\partial_j Q(\cdot-y)\rangle = -\dot y_k\Re\langle g,\partial_{jk}Q(\cdot-y)\rangle,
\]
where the repeated index $k$ is summed. In particular, this term is $\mathcal{O}(|\dot y|\,\|g\|_{H^1}),$ and hence continuing from above we may derive the desired estimate
\[
|\dot y| \lesssim \|g\|_{H^1}. 
\]
\end{proof}

\begin{proof}[Proof of Proposition~\ref{P:modulation2}] Combining Lemmas~\ref{L:modulation-orthogonality}, \ref{L:boundsI}, \ref{L:boundsII} and \ref{L:boundsIII}, we immediately obtain the desired decomposition and bounds. 
\end{proof}



\end{document}